\documentclass{article}
\usepackage{amsfonts, amsmath, amssymb, amscd, amsthm, graphicx,enumerate}
\usepackage{hyperref,xypic}
\usepackage[usenames]{color}
\usepackage{tikz}
\usepackage{braket}
\usepackage{float}
\usepackage[caption = false]{subfig}
\usepackage{overpic}

\usetikzlibrary{cd}
\usepackage{thmtools} 
\usepackage{thm-restate}

\makeatletter
\def\blfootnote{\xdef\@thefnmark{}\@footnotetext}
\makeatother

\newtheorem{thm}{Theorem}[section]
\newtheorem{cor}[thm]{Corollary}
\newtheorem{lem}[thm]{Lemma}
\newtheorem{prop}[thm]{Proposition}

\theoremstyle{definition}
\newtheorem{defn}[thm]{Definition}
\theoremstyle{remark}
\newtheorem{rem}[thm]{Remark}
\newtheorem{ex}[thm]{Example}

\newtheorem{claim}[thm]{Claim}

\newfont{\eufm}{eufm10}

\renewcommand{\phi}{\varphi}

\newcommand{\Z}{\mathbb Z}
\newcommand{\R}{\mathbb R}
\newcommand{\Q}{\mathbb Q}

\newcommand{\Aut}{\operatorname{Aut}}

\renewcommand{\H}{\mathcal{H}}

\newcommand{\hyp}{\mathbb{H}}

\def\mc {\mathcal}
\def\onto {\twoheadrightarrow}

\begin{document}

\title{Actions of solvable Baumslag-Solitar groups on hyperbolic metric spaces}
\author{Carolyn R. Abbott   \and Alexander J. Rasmussen}

\date{}
\maketitle

\begin{abstract}
We give a complete list of the cobounded actions of solvable Baumslag-Solitar groups on hyperbolic metric spaces up to a natural equivalence relation. The set of equivalence classes  carries a natural partial order first introduced by Abbott--Balasubramanya--Osin, and we describe the resulting poset completely. There are finitely many equivalence classes of actions, and each equivalence class contains the action on a point, a tree, or the hyperbolic plane.
\end{abstract}

%\tableofcontents

%%%%%%%%%%%%%%%%%%%%%%%%%%%%%%%%%%%%%%%%%%%%%%%%%%%%%%%%%%%%%%%%%%%%%%%%%%%%%%%%%%%%%%%%%%%%%%%%%%%%%%%%%%%%%%%%%%%%%%

\section{Introduction}

The Baumslag-Solitar groups $BS(m,n)$ are a classically-studied family of groups defined by the particularly straightforward presentations $\langle a, t \mid ta^mt^{-1}=a^n\rangle$. In the case $m=n=1$, $BS(1,1)$ is isomorphic to the abelian group $\Z^2$. In general, for $n\geq 2$, $BS(1,n)$ is a nonabelian solvable group via the isomorphism $BS(1,n)\cong \Z\left[\frac{1}{n}\right]\rtimes \Z$.

The group $BS(1,n)$ admits several natural actions on hyperbolic metric spaces. The Cayley graph of $BS(1,n)$ with respect to the generating set $\{a^{\pm 1},t^{\pm 1}\}$ consists of a number of ``sheets,'' each of which is quasi-isometric to the hyperbolic plane $\hyp^2$. The sheets are glued together along complements of horoballs in a pattern described by the $(n+1)$-regular tree. The result is pictured in Figure \ref{proj} for the case $n=2$. Collapsing each sheet down to a vertical geodesic line gives a projection from the Cayley graph to the $(n+1)$-regular tree. Moreover, the action of $BS(1,n)$ on its Cayley graph permutes the fibers of this projection, so that we obtain an action of $BS(1,n)$ on the $(n+1)$-regular tree. Similarly, the idea of ``collapsing the sheets down to a single sheet'' gives an action on the hyperbolic plane, although this idea is a bit harder to formalize.

\begin{figure}[h]
\label{proj}
\begin{center}
\includegraphics[width=0.7\textwidth]{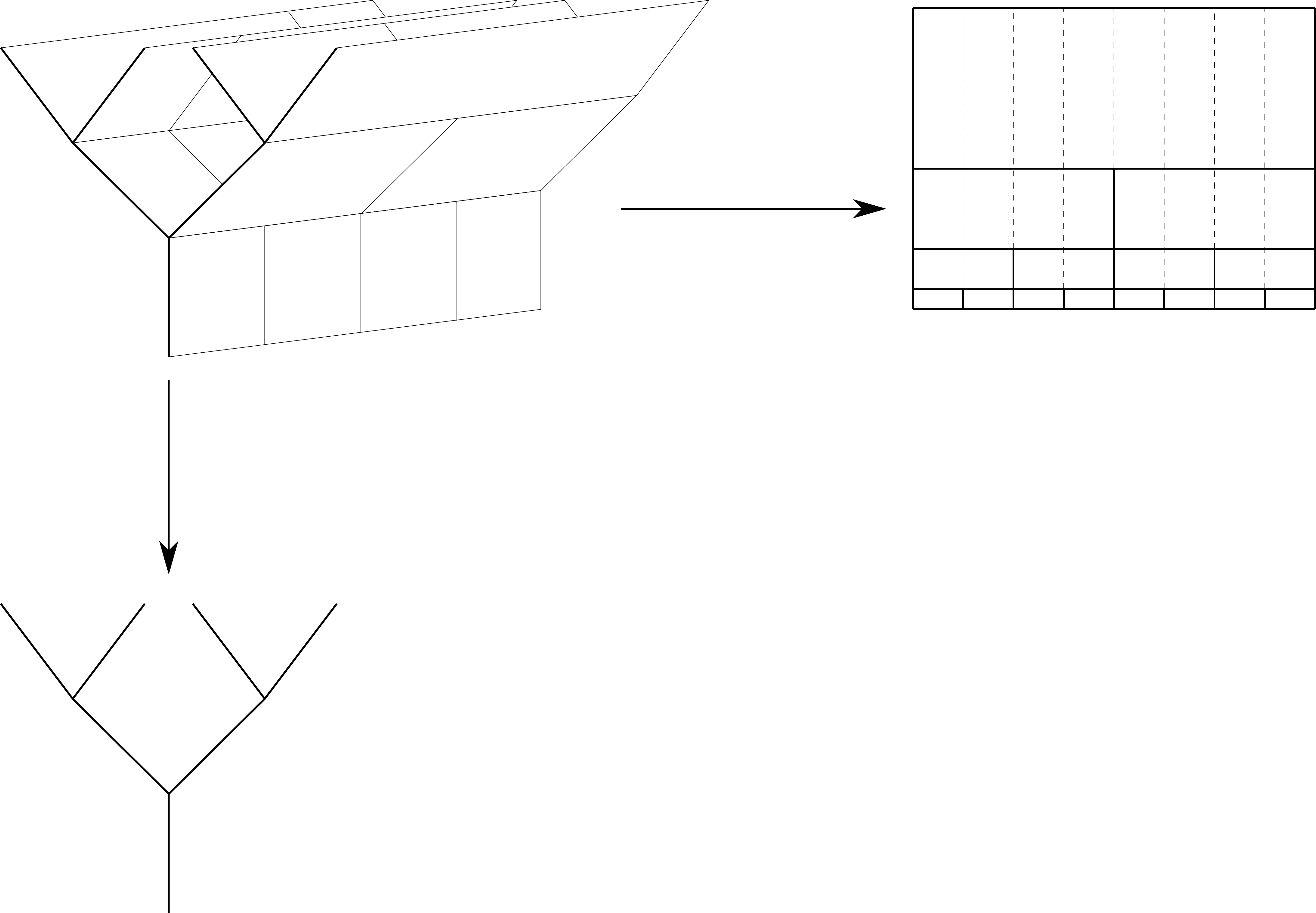}
\caption{Two natural actions of the group $BS(1,2)$ via projections.}
\end{center}
\end{figure}

Formally, the action of $BS(1,n)$ on $\hyp^2$ is given by the representation \[BS(1,n)\to \operatorname{PSL}(2,\R),\] where  \[ a\mapsto \begin{pmatrix} 1 & 1 \\ 0 & 1 \end{pmatrix}, \quad t \mapsto \begin{pmatrix} \sqrt{n} & 0 \\ 0 & 1/\sqrt{n} \end{pmatrix}.\] Another way to obtain the natural action on the $(n+1)$-regular tree is by expressing $BS(1,n)$ as an HNN extension over the subgroup $\langle a \rangle \cong \Z$ and considering the action of $BS(1,n)$ on the Bass-Serre tree of the resulting one-edge graph of groups.

In addition to these actions, $BS(1,n)$ admits an obvious homomorphism $BS(1,n)\to \Z$ defined by $a\mapsto 0$, $t\mapsto 1$. This defines an action of $BS(1,n)$ on (the hyperbolic metric space) $\R$ via integral translations. This action may also be obtained by collapsing either the hyperbolic plane or the $(n+1)$-regular tree onto a vertical geodesic in a height-respecting manner and noting that $BS(1,n)$ permutes the fibers of the resulting projection. Even more trivially, any group admits an action on a point (which is a hyperbolic metric space).

All of these actions are \textit{cobounded} in the sense that the orbit of a point under the action admits the entire space as a bounded neighborhood. One may naturally wonder whether these four actions (the actions on a tree, the hyperbolic plane, the line, and a point) give a complete list of the nontrivial cobounded actions of $BS(1,n)$ on hyperbolic metric spaces up to equivalence. In this paper we show that this is indeed the case \textit{if} $n$ is prime.

In the case that $n$ is \emph{not} prime, we show that $BS(1,n)$ admits actions on certain other Bass-Serre trees which may be understood algebraically. For each divisor $l$ of $n$, $\Z\left[\frac{1}{l}\right]$ is a subring of $\Z\left[\frac{1}{n}\right]$. We may form an ascending HNN extension of $\Z\left[\frac{1}{l}\right]$. Specifically, consider the one-edge graph of groups with vertex group $\Z\left[\frac{1}{l}\right]$ and edge group $\Z\left[\frac{1}{l}\right]$ which includes isomorphically onto $\Z\left[\frac{1}{l}\right]$ on one end and as the subgroup $n\Z\left[\frac{1}{l}\right]$ on the other end. It is not too hard to show that the fundamental group of this graph of groups is $BS(1,n)$ and thus there is an action of $BS(1,n)$ on the corresponding Bass-Serre tree. Considering these actions on Bass-Serre trees together with the canonical actions described in the above paragraph, we show that this gives a complete list of the cobounded hyperbolic actions of $BS(1,n)$ (up to an equivalence relation, described below). Before stating these results precisely, we need to introduce some terminology.

\subsection{Hyperbolic structures on groups}

In groups that admit interesting actions on hyperbolic metric spaces, it is natural to wonder whether it is possible to describe all such actions explicitly. Unfortunately, this (slightly naive) goal is currently unattainable for almost all commonly studied groups. For instance, using the machinery of combinatorial horoballs introduced by Groves--Manning in \cite{GM}, one may produce uncountably many \textit{parabolic} actions of any countable group on hyperbolic metric spaces, i.e. actions with a fixed point on the boundary and all group elements acting elliptically or parabolically. For this reason, we restrict to considering \textit{cobounded} actions on hyperbolic metric spaces, which are never parabolic. Moreover, we must use some notion of \textit{equivalence} for the actions of a fixed group on different hyperbolic spaces, as it is quite easy to modify an action on a given hyperbolic space equivariantly to produce an action on a quasi-isometric space. The equivalence relation we consider, introduced in \cite{ABO}, is roughly \textit{equivariant quasi-isometry}. See Section \ref{section:background} for more details.

Having restricted to cobounded actions up to equivalence, it is usually still quite difficult to describe explicitly all of the equivalence classes of actions of a given group on different hyperbolic metric spaces. For instance, in \cite{ABO} Abbott--Balasubramanya--Osin considered the hyperbolic actions of \textit{acylindrically hyperbolic} groups, a very wide class of groups all displaying some features of negative curvature. There they showed that any acylindrically hyperbolic group admits uncountably many distinct equivalence classes of actions on hyperbolic spaces.

But for groups which don't display strong features of negative curvature, it may be possible to give a complete list of their cobounded hyperbolic actions. For instance, in \cite{ABO} Abbott--Balasubramanya--Osin give a complete list of the equivalence classes of cobounded actions of $\Z^n$ on hyperbolic metric spaces. More recently, in \cite{Bal}, Balasubramanya gives a complete list of the actions of \textit{lamplighter groups} on hyperbolic spaces. Our work draws inspiration from her strategy.

In these cases, it is possible to say even more about the actions of a fixed group on different hyperbolic metric spaces. We are interested in the question of when one action ``retains more information'' about the group than another. This question leads to a partial order on the set of equivalence classes of actions of a group $G$ on hyperbolic spaces $X$, defined in \cite{ABO}. Roughly, we say that $G\curvearrowright X$ \emph{dominates} $G\curvearrowright Y$ when the action $G\curvearrowright Y$ may be obtained by equivariantly \textit{coning off} certain subspaces of $X$. See Section \ref{sec:qpar} for the precise definition. This partial order descends to a partial order on equivalence classes of actions.

Hence, for a group $G$, the set of equivalence classes of actions of $G$ on different hyperbolic metric spaces is a poset $\mathcal{H}(G)$. In this paper we give a complete description of the poset $\mathcal H(BS(1,n))$. Let $n=p_1^{n_1}p_2^{n_2}\dots p_k^{n_k}$ be the prime factorization of $n$ and let $\mathcal K_n$ be the poset $2^{\{1,\dots,k\}}\setminus\{\emptyset\}$, with the partial order given by inclusion. 
\begin{thm}\label{thm:BS1nstructure}
For any $n\geq 2$, $\mathcal H(BS(1,n))$ has the following structure: $\mathcal H_{qp}(BS(1,n))$, the subposet of quasi-parabolic structures, consists of a copy of $\mathcal K_n$ and a single additional structure which is incomparable to every element of $\mathcal K_n$; every quasi-parabolic structure dominates a single lineal structure, which dominates a single elliptic structure (see Figure \ref{fig:BS1nposet}).
\end{thm}

\begin{figure}[h!]

\centering

\begin{tikzpicture}[scale=0.7]

\node[circle, draw, minimum size=1.25cm] (triv) at (0,-3) {$\curvearrowright *$};
\node[circle, draw, minimum size=1.25cm] (lin) at (0,0) {$\curvearrowright \R$};
\node[circle, draw, minimum size=1.25cm] (hyp) at (-3,3) {$\curvearrowright \hyp^2$};
\node[circle, draw, minimum size=1.25cm] (bs) at (3,8) {$\curvearrowright T$};

\node (low1) at (1.5,3) {};
\node (low2) at (2,3) {};
\node (low3) at (2.5,3) {$\dots$};
\node (low4) at (3,3) {};
\node (low5) at (3.5,3) {};

\node (up1) at (1.5,5) {};
\node (up2) at (2,5) {};
\node (up3) at (2.5,5) {$\dots$};
\node (up4) at (3,5) {};
\node (up5) at (3.5,5) {};

\node[red] (poset) at (5.7,4) {$2^{\{1,\ldots,k\}}$};

\draw[thick] (triv) -- (lin) -- (hyp);
\draw[thick] (lin) -- (low1);
\draw[thick] (lin) -- (low2);
\draw[thick] (lin) -- (low4);
\draw[thick] (lin) -- (low5);

\draw[thick] (up1) -- (bs);
\draw[thick] (up2) -- (bs);
\draw[thick] (up4) -- (bs);
\draw[thick] (up5) -- (bs);

\draw[thick, dotted, red, rotate=-20] (0.2, 4.3) ellipse (80pt and 150pt); 
\end{tikzpicture}

\caption{The poset $\mathcal{H}(BS(1,n))$. The subposet circled in red is isomorphic to the power set $2^{\{1,\dots, k\}}$, which is a lattice.}
 \label{fig:BS1nposet}
\end{figure}

Following \cite{ABO}, we say a group $G$ is \emph{$\mc H$--accessible} if $\mc H(G)$ contains a largest element. Otherwise, we say $G$ is \emph{$\mc H$--inaccessible}.  The following result is immediate.

\begin{cor}
$BS(1,n)$ is $\mathcal H$--inaccessible.
\end{cor}

Farb--Mosher  prove in \cite{FM1} that two solvable Baumslag-Solitar groups $BS(1,m)$ and $BS(1,n)$ with $m,n\geq 2$ are quasi-isometric if and only if they are commensurable, which occurs if and only if there exist integers $r,i,j>0$ such that $m=r^i$ and $n=r^j$.  In this case, we have $\mc K_n=\mc K_m$, which yields the following corollary.

\begin{cor} \label{cor:qirigidity}
If $BS(1,m)$ and $BS(1,n)$ are quasi-isometric, then the posets $\mc H(BS(1,m))$ and $\mc H(BS(1,n))$ are isomorphic.
\end{cor}

Theorem \ref{thm:BS1nstructure} may be related to the algebraic description of Bass-Serre trees described above. Recall that for each divisor $l$ of $n$ there is a Bass-Serre tree of $BS(1,n)$ corresponding to the subring $\Z\left[\frac{1}{l}\right]$ of $\Z\left[\frac{1}{n}\right]$. Note moreover that if the integers $l$ and $m$ both have the same prime divisors then in fact $\Z\left[\frac{1}{l}\right]=\Z\left[\frac{1}{m}\right]$. Thus the subring and the Bass-Serre tree only depend on the set of prime divisors of $l$. We obtain one Bass-Serre tree for each subset of $\{p_1,\ldots,p_k\}$, and these correspond exactly to the subposet $\mc K_n$ of $\mc H(BS(1,n))$.

\subsection{About the proof}

Any group action on a hyperbolic space falls into one of finitely many types (elliptic, parabolic, lineal, quasi-parabolic, or general type) depending on the number of fixed points on the boundary and the types of isometries defined by various group elements. See Section \ref{section:background} for precise definitions.

Since $BS(1,n)$ is solvable, it contains no free subgroup and therefore by the Ping-Pong Lemma, any non-elliptic action of $BS(1,n)$ on a hyperbolic metric space must have a fixed point on the boundary. Since we consider only cobounded actions, we see that $BS(1,n)$ can have only lineal or quasi-parabolic cobounded actions on hyperbolic metric spaces. Hence we are left to consider such actions.

We crucially use the fact that $BS(1,n)$ can be written as a semidirect product $\Z\left[\frac{1}{n}\right]\rtimes_\alpha \Z$. In \cite{CCMT} Caprace--Cornulier--Monod--Tessera classified the lineal and quasi-parabolic actions of certain groups $H\rtimes \Z$ in the language of \textit{confining subsets} of $H$ under the action of $\Z$ (see Section \ref{section:background}).

Using techniques developed in \cite{CCMT}, we show that the lineal and quasi-parabolic actions of $BS(1,n)$ naturally correspond to confining subsets of $\Z\left[\frac{1}{n}\right]$ under the actions of $\alpha$ and $\alpha^{-1}$, and so we are led to try to classify such subsets. The confining subsets under the action of $\alpha^{-1}$ are straightforward to classify, and in fact they all correspond to (the equivalence class of) the action of $BS(1,n)$ on $\mathbb H^2$. On the other hand, the classification of confining subsets under the action of $\alpha$ is more complicated. We show that such subsets correspond in a natural way to ideals in the ring of $n$-adic integers $\Z_n$. To see how such ideals arise, we consider confining subsets $Q\subset \Z\left[\frac{1}{n}\right]$ under the action of $\alpha$ and  write elements  $a\in Q$ in base $n$: \[a=\pm a_r\cdots a_0.a_{-1} \cdots a_{-s},\] allowing any number of 0's at the end of this expression. We then consider the set of all $\ldots x_2 x_1 x_0\in \Z_n$ such that for any $s$, the sequence $x_sx_{s-1}\ldots x_0$ appears as the sequence of digits to the right of the decimal point in some element of $Q$. That is, we require that there is a number of the form \[a_r \cdots a_0 .x_s \cdots x_0\] in $Q$ for arbitrarily large $s$. We show that the set of $n$-adic integers $\ldots x_s x_{s-1} \ldots x_0$ obtained in this manner is an ideal of $\Z_n$. It may be thought of as a kind of \emph{limit set} of $Q$. We also show that this process can be reversed to associate confining subsets of $\Z[\frac1n]$ to ideals of $\Z_n$. With this correspondence in hand, we describe how all of the resulting actions are equivalent to the actions of $BS(1,n)$ on certain Bass-Serre trees.

\subsection{Other results and results in the literature}

The poset described in Theorem \ref{thm:BS1nstructure} is interesting because of its asymmetry (when $n$ is not a power of a prime). In \cite{Bal} Balasubramanya described $\mathcal{H}(L_n)$ where $L_n$ is the Lamplighter group $\Z/n\Z \wr \Z=\Z/n\Z \wr \langle t \rangle$. In this case, $\mathcal{H}(L_n)$ splits into two isomorphic sub-posets of quasi-parabolic structures (corresponding to actions in which the fixed point of $L_n$ is the attracting fixed point of $t$, and respectively the repelling fixed point of $t$) which each dominate a single lineal structure which in turn dominates an elliptic structure. Work of the authors in \cite{AR} shows that this structure also holds for semidirect products $\Z^2 \rtimes_\alpha \Z$ where $\alpha\in SL(2,\Z)$ is an Anosov matrix. It would be interesting to see what extra properties of a semidirect product $G\rtimes \Z$ are needed to ensure this kind of symmetry.

In \cite{Bal} Balasubramanya considers hyperbolic actions of general wreath products $G \wr \Z=\left(\bigoplus_{n \in \Z} G\right) \rtimes \Z$ and shows that $\mathcal{H}(G\wr \Z)$ always contains two copies of the poset of subgroups of $G$. In the case $G=\Z/n\Z$ we have $G\wr \Z=L_n$ and this suffices to describe all of $\mathcal{H}(L_n)$.

In Section \ref{section:generating} we describe a general algebraic construction of quasi-parabolic structures on semidirect products $H\rtimes \Z$. We show that in the case of $\Z \wr \Z$, this construction suffices to produce a countable chain of quasi-parabolic structures.

\noindent{\bf Acknowledgements:}  The first author was partially supported by  NSF Award DMS-1803368.  The second author was partially supported by NSF Award DMS-1610827.  The authors thank Denis Osin for pointing out Corollary \ref{cor:qirigidity} and the anonymous referee for helpful comments.

%%%%%%%%%%%%%%%%%%%%%%%%%%%%%
\section{Background}  \label{section:background}

\subsection{Actions on hyperbolic spaces} \label{sec:actions}

Given a metric space $X$, we denote by $d_X$ the distance function on $X$. A map $f\colon X\to Y$ between metric spaces $X$ and $Y$ is a \emph{quasi-isometric embedding} if there is a constant $C$ such that for all $x,y\in X$, \[\frac1C d_X(x,y)-C\leq d_Y(f(x),f(y))\leq  Cd_X(x,y)+C.\]  If, in addition, $Y$ is contained in the $C$--neighborhood of $f(X)$ then $f$ is called a \emph{quasi-isometry}. If a group $G$ acts (by isometries) on $X$ and $Y$, then a map $f\colon X\to Y$ is \emph{coarsely $G$--equivariant} if for every $x\in X$ we have \[\sup_{g\in G} d_Y(f(gx),gf(x))<\infty.\]

We will assume that all actions are by isometries. The action of a group $G$ on a metric space $X$ is \emph{cobounded} if there exists a bounded diameter subspace $B\subseteq X$ such that $X=\bigcup_{g\in G}gB$.

Given an action $G\curvearrowright X$ of $G$ on a hyperbolic space, an element $g\in G$ is \emph{elliptic} if it has bounded orbits; \emph{loxodromic} if the map $\Z\to X$ given by $n\mapsto g^n\cdot x_0$ for some (equivalently, any) $x_0\in X$ is a quasi-isometric embedding; and \emph{parabolic} otherwise.

Any group action on a hyperbolic space falls into one of finitely many types depending on the number of fixed points on the boundary and the types of isometries defined by various group elements. This classification was described by Gromov in \cite{Gro}: the action $G\curvearrowright X$ (where $X$ is hyperbolic) is
\begin{itemize}
\item \textit{elliptic} if $G$ has a bounded orbit in $X$;
\item \textit{lineal} if $G$ fixes two points in $\partial X$;
\item \textit{parabolic} if $G$ fixes a unique point of $\partial X$ and no element of $G$ acts as a loxodromic isometry of $X$;
\item \textit{quasi-parabolic} if $G$ fixes a unique point of $\partial X$ and at least one element of $G$ acts as a loxodromic isometry; and
\item \textit{general type} if $G$ doesn't fix any point of $\partial X$ and at least one element of $G$ acts as a loxodromic isometry.
\end{itemize}

\subsection{Hyperbolic structures} \label{sec:qpar}

Fix a group $G$. For any (possibly infinite) generating set $S$ of $G$,  let $\Gamma(G,S)$ be the Cayley graph of $G$ with respect to the generating set $S$, and let $\|\cdot\|_S$ denote the word norm on $G$ with respect to $S$.
Given two generating sets $S,T$ of a group $G$, we say $T$ is \emph{dominated by} $S$, written  $T\preceq S$, if \[\sup_{g\in S}\|g\|_T<\infty.\]  It is clear that $\preceq$ is a preorder on the set of generating sets of $G$ and so induces an equivalence relation: $S\sim T$ if and only if $T\preceq S$ and $S\preceq T$. Let $[S]$ be the equivalence class of a generating set. Then the preorder $\preceq$ induces a partial order  $\preccurlyeq$ on the set of all equivalence classes of generating sets of $G$ via $[S]\preccurlyeq [T]$ if and only if $S\preceq T$.

\begin{defn}
Given a group $G$, the \emph{poset of hyperbolic structures on $G$} is defined to be \[\mathcal H(G):= \{[S]\mid G=\langle S\rangle \textrm{ and } \Gamma(G,S) \textrm{ is hyperbolic}\},\] equipped with the partial order $\preccurlyeq$.
\end{defn}

Notice that since hyperbolicity is a quasi-isometry invariant of geodesic metric spaces, the above definition is independent of the choice of representative of the equivalence class $[S]$. Every element $[S]\in\mc H(G)$ gives rise to a cobounded action on a hyperbolic space, namely $G\curvearrowright \Gamma(G,S)$. Moreover, given a cobounded action on a hyperbolic space $G\curvearrowright X$, a standard Schwarz--Milnor argument (see \cite[Lemma~3.11]{ABO}) provides a generating set $S$ of $G$ such that $\Gamma(G,S)$ is equivariantly quasi-isometric to $X$. We say that two actions $G\curvearrowright X$ and $G\curvearrowright Y$ are \emph{equivalent} if there exists a coarsely $G$--equivariant quasi-isometry $X\to Y$. By \cite[Proposition~3.12]{ABO}, there is a one-to-one correspondence between equivalence classes $[S]\in\mathcal H(G)$ and equivalence classes of cobounded actions $G\curvearrowright X$ with $X$ hyperbolic.

We denote the set of equivalence classes of cobounded elliptic, lineal, quasi-parabolic, and general-type actions by $\mc H_e,\mc H_\ell,\mc H_{qp},$ and $\mc H_{gt}$, respectively. Since parabolic actions cannot be cobounded, we have for any group $G$, \[\mc H(G)=\mc H_e(G)\sqcup\mc H_\ell(G)\sqcup\mc H_{qp}(G)\sqcup\mc H_{gt}(G).\]

A lineal action of a group $G$ on a hyperbolic space $X$ is \emph{orientable} if no element of $G$ permutes the two limit points of $G$ on $\partial X$. We denote the set of equivalence classes of orientable lineal actions of $G$ by $\mc H_\ell^+(G)$.

The action of a group $G$ on a hyperbolic metric space $X$ is \emph{focal} if it fixes a boundary point $\xi\in\partial X$ and if some element of $G$ acts as a loxodromic isometry. If $[S]\in\mathcal H(G)$ is a focal action, then $[S]\in \mathcal H_\ell^+(G)\sqcup\mathcal H_{qp}(G)$.

\noindent {\bf Quasi-characters.}  A map $q\colon G\to \R$ is a \emph{quasi-character} (also called a \emph{quasimorphism}) if there exists a constant $D$ such that for all $g,h\in G$, we have $|q(gh)-q(g)-q(h)|\leq D$. We say that $q$ has \emph{defect at most $D$}. If, in addition, the restriction of $q$ to every cyclic subgroup is a homomorphism, then $q$ is called a \emph{pseudocharacter} (or \emph{homogeneous quasimorphism}). Every quasi-character $q$ gives rise to a pseudocharacter $\rho$ defined by $\rho(g)=\lim_{n\to\infty}\frac{q(g^n)}{n}$; we call $\rho$ the \emph{homogenization of $q$}. Every pseudocharacter is constant on conjugacy classes. If $q$ has defect at most $D$, then it is straightforward to check that $|q(g)-\rho(g)|\leq D$ for all $g\in G$.

Let $G\curvearrowright X$ be an action on a hyperbolic space with a global fixed point $\xi\in\partial X$. For any sequence $\mathbf x=(x_n)$  in $X$ converging to $\xi$ and any fixed basepoint $s\in X$, we define the associated quasi-character $q_{\bf x}\colon G\to \mathbb R$ as follows. For all $g\in G$, \begin{equation} \label{eqn:qchar} q_\mathbf x(g)=\limsup_{n\to\infty}(d_X(gs,x_n)-d_X(s,x_n)). \end{equation}

Its homogenization $\rho_\mathbf x\colon G\to\mathbb R$ is the \emph{Busemann pseudocharacter}. It is known that for any two sequences $\mathbf x,\mathbf y$ converging to $\xi$, $\sup_{g\in G}|q_\mathbf x(g)-q_\mathbf y(g)|<\infty$, and thus we may drop the subscript $\mathbf x$ in $\rho_\mathbf x$. If  $\rho$ is a homomorphism, then the action $G\curvearrowright X$ is called \emph{regular}.

\begin{lem}[{\cite[Lemma~3.8]{CCMT}}]
Let $G\curvearrowright X$ be an action on a hyperbolic space with a global fixed point $\xi\in\partial X$. Then the (possibly empty) set of loxodromic isometries of the action is $\{g\in G\mid \rho(g)\neq 0\}$, and the set of those with attracting fixed point $\xi$ is $\{g\in G\mid \rho(g)<0\}$. In particular, the action of $G$ is elliptic or parabolic if and only if $\rho\equiv 0$, and  lineal or quasi-parabolic otherwise.
\end{lem}

\subsection{Confining subsets}\label{sec:confining}

Consider a group $G=H\rtimes_{\alpha} \Z$ where $\alpha\in \Aut(H)$ acts by $\alpha(h)=tht^{-1}$ for any $h\in H$, where $t$ is a generator of $\Z$. Let $Q$ be a \textit{symmetric} subset of $H$.  The following definition is from \cite[Section~4]{CCMT}.

\begin{defn} \label{def:confining}
The action of $\alpha$ is \emph{(strictly) confining $H$ into $Q$} if it satisfies the following three conditions.
\begin{enumerate}[(a)]
\item $\alpha(Q)$ is (strictly) contained in $Q$;
\item $H=\bigcup_{k\geq 0} \alpha^{-k}(Q)$; and
\item $\alpha^{k_0}(Q\cdot Q)\subseteq Q$ for some $k_0\in \Z_{\geq 0}$.
\end{enumerate}
\end{defn}

\begin{rem}
The definition of confining subset given in \cite{CCMT} does not require symmetry of the subset $Q\subset H$. However, according to \cite[Theorem~4.1]{CCMT}, to classify regular quasi-parabolic structures on a group it suffices to consider only confining subsets which are symmetric. See also Proposition \ref{prop:qpconfining} in this paper.
\end{rem}

\begin{rem} By the discussion after the statement of \cite[Theorem~4.1]{CCMT}, if there is a subset $Q\subseteq H$ such that the action of $\alpha$ is confining $H$ into $Q$ but not strictly confining, then $[Q\cup\{t^{\pm 1}\}]\in \mathcal H_\ell^+(G)$. If the action is strictly confining, then $[Q\cup\{t^{\pm 1}\}]\in \mathcal H_{qp}(G)$.
\end{rem}

In this paper, we will focus primarily on describing subsets $Q$ of $H$ into which the action of $\alpha$ is (strictly) confining $H$. For brevity, we will refer to such $Q$ as \emph{(strictly) confining under the action of $\alpha$}, or simply \emph{(strictly) confining} if the action of $\alpha$ is understood.

To see an example of how confining subsets arise, it is useful to consider the action $BS(1,n)\curvearrowright \hyp^2$ described in the introduction. In this action, the conjugates of $t$ act as loxodromic isometries whereas the elements of $H=\Z\left[\frac{1}{n}\right]$ act as parabolic isometries. Consider the subset $Q\subset H$ of isometries that translate a given point $p$ (say $i$ in the upper half plane model) by some bounded distance (say 1). If $g\in H$ then we consider the action of a conjugate $t^{-k}gt^k$ where $k\gg 0$. Considering the actions of $t^k$, $g$, and $t^{-k}$ in turn we see that:
\begin{itemize}
\item $t^k$ first translates $p$ vertically by a very large distance,
\item $g$ shifts $t^kp$ within the horocycle based at $\infty$ containing it by a very small distance,
\item $t^{-k}$ takes this horocycle isometrically back to the horocycle containing $p$.
\end{itemize} In other words, $t^{-k}gt^k$ is a parabolic isometry which moves $p$ by a much smaller distance than $g$ itself does. In particular, if $k$ is large enough, then $t^{-k}gt^k\in Q$. Furthermore, the infimal $k$ with $t^{-k}gt^k\in Q$ depends only on how far $g$ moves the original point $p$. Using these facts it is easy to see that $Q$ is a confining subset of $H$.

Thus there is a correspondence of quasi-parabolic structures on $BS(1,n)$ and confining subsets of $H$. Precisely, we have the following result, which is a minor modification of \cite[Theorems 4.4 \& 4.5]{Bal} and \cite[Proposition~4.5]{CCMT}. It is proved in Section \ref{section:qpstruct}.

\begin{restatable}{prop}{qpconfining}
\label{prop:qpconfining}
A hyperbolic structure $[T]$ is an element of $\mathcal{H}_{qp}(BS(1,n))$ if and only if there exists a symmetric subset $Q\subset \Z\left[\frac{1}{n}\right]$ which is strictly confining under the action of $\alpha$ or $\alpha^{-1}$ such that $[T]=[Q\cup\{t^{\pm 1}\}]$.
\end{restatable}

\subsection{$n$--adic integers}

\begin{defn}
An \emph{$n$--adic integer} is an infinite series \[\sum_{i=0}^\infty a_in^i,\] where $a_i\in\{0,1,\dots,n-1\}$. Such an element can also be written in its base $n$ expansion as \[\dots a_3a_2a_1a_0.\] We denote the set of $n$--adic integers by $\Z_n$.
\end{defn}

We define operations of addition and multiplication on $\Z_n$, which gives it the structure of a ring.

\begin{defn}
Let $a=\sum_{i=0}^\infty a_in^i$ and $b=\sum_{i=0}^\infty b_in^i$ be elements of $\Z_n$. Then the sum $a+b$ is the $n$--adic integer $c=\sum_{i=0}^\infty c_in^i$ defined inductively as follows. Let $c_0=a_0+b_0\mod n$ (where we identify $\Z/n\Z$ with $\{0,\dots,n-1\}$) and $t_0=\left\lfloor \frac{a_0+b_0}{n}\right\rfloor$, so that $a_0+b_0=c_0+t_0n$. Assume that $c_0,\dots, c_{i-1}$ and $t_0,\dots,t_{i-1}$ have been defined, and let \[c_i=a_i+b_i+t_{i-1}\mod n\] and \[t_i= \left\lfloor\frac{a_i+b_i+t_{i-1}}{n}\right\rfloor,\]
so that $a_i+b_i=c_i+t_in$.

The product $ab$ is the $n$--adic number $d=\sum_{i=o}^\infty d_in^i$ defined inductively as follows. Let $d_0=a_0b_0\mod n$, and let $s_0=\left\lfloor\frac{a_0b_0}{n}\right\rfloor$, so that $a_0b_0=d_0+s_0n$. Assume that $d_1,\dots, d_{i-1}$ and $s_1,\dots, s_{i-1}$ have been defined, and let \[d_i=\sum_{j=0}^ia_jb_{i-j}+s_{i-1}\mod n\] and \[s_i=\left\lfloor\frac{\sum_{j=0}^ia_jb_{n-j}+s_{i-1}}{n}\right\rfloor,\] so that $\sum_{j=0}^ia_jb_{n-j}+s_{i-1}=d_i+s_in.$
\end{defn}

In the above operations, we think of $t_i$ and $s_i$ as the amounts that are ``carried" at each step, analogous to how we carry digits when adding and multiplying in base 10. An element $a=\dots a_2a_1a_0$ of the ring $\Z_n$ is a unit if and only if $a_0$ is relatively prime to $n$.

Let 
\begin{equation}\label{eqn:phil}
\phi_l\colon \Z_n\to\Z/n^l\Z
\end{equation} 
be the ring homomorphism which identifies an element of $\Z_n$ with its $l$--th partial sum: \[\phi_l(a)=\phi_l\left(\sum_{i=0}^\infty a_in^i\right)=\sum_{i=0}^{l-1} a_in^i.\]  These homomorphisms are compatible in the following sense. For any $k\leq l$, let the map \[{}_n^{}T^l_k\colon \Z/n^l\Z\to \Z/n^k\Z\] be reduction modulo $n^k$. Then for any $a\in \Z_n$, we have ${}_n^{}T^l_k(\phi_l(a))=\phi_k(a)$.

In fact, any infinite sequence $(a^i)$ of elements $a^i\in\Z/n^i\Z$  satisfying ${}_nT^l_k(a^l)=a^k$ defines an element $a\in \Z_n$. Namely, identify $a^l$ with the unique representative $b^l$ of its congruence class in the set $ \{0,1,\ldots,n^l-1\}$. We may write $b^l=\sum_{i=0}^{l-1} b^l_i n^i$.  Then for any $i\geq 0$ and $l,k>i$ we have $b^l_i=b^k_i$, and hence we may write \[a_0=b^1_0=b^2_0=\cdots, \quad a_1=b^2_1=b^3_1=\cdots,\quad \ldots.\] Writing \[a=\sum_{i=0}^\infty a_i n^i\] we have $\phi_i(a)=a^i$ for all $i$. In particular, this shows that $\Z_n$ is isomorphic to the inverse limit $\varprojlim \Z/n^i\Z$.

There is a metric $d$ on $\Z_n$, called the \emph{$n$--adic metric}, defined by $d(x,y)= n^{-q}$ if $q$ is maximal such that $n^q$ divides $x-y$, i.e., if the first $q$ digits of $x-y$ are zero and the $(q+1)$--st digit is non-zero.

\begin{lem} 
With the topology coming from the $n$--adic metric, $\Z_n$ is a compact space. 
\end{lem}

\begin{proof}
Suppose we have an infinite sequence $\{x^i\}_{i\in\Z_{>0}}$ such that for each $i$, we have $x^i=\dots x^i_3x^i_2x^i_1x^i_0$. By the pigeon-hole principle, there is some $y_0\in\{0,\dots, n-1\}$ such that $x^i_0=y_0$ for infinitely many $i$. The collection of these $x^i$ forms a subsequence $\{x^{i_{0,j}}\}_{j \in \Z_{>0}}$. Repeating this construction iteratively, we have a sequence of subsequences $\{\{x^{i_{k,j}}\}_j\}_k$ and a number $y=\dots y_3y_2y_1y_0\in\Z_n$ such that for each $k$ every element of $\{x^{i_{k,j}}\}_j$ agrees with $y$ in its first $k+1$ digits. Moreover, $\{x^{i_{k+1,j}}\}_j$ is a subsequence of $\{x^{i_{k,j}}\}_j$. Thus the diagonal subsequence $\{x^{i_{k,k}}\}_k$ of $\{x^i\}_i$ converges to $y$. Consequently, $\Z_n$ is sequentially compact. Since $\Z_n$ is a metric space, it is also compact.
\end{proof}

We now consider the ring structure of $\Z_n$.

\begin{lem} \label{lem:remainderthm}
Let $n=p_1^{n_1}p_2^{n_2}\dots p_k^{n_k}$. There is an isomorphism \begin{equation} \label{eqn:isomorphism}\Z_n\cong \Z_{p_1^{n_1}}\times\cdots\times \Z_{p_k^{n_k}}.\end{equation} 
\end{lem}

Let us give a description of this isomorphism. First, suppose $n=n_1n_2$, where $n_1$ and $n_2$ are relatively prime. To define a map on $\Z_n$, we use the identification of an element of $a\in\Z_m$ with the sequence $(\phi_l(a))\in \varprojlim\Z/m^l\Z$. Let  \[f\colon\Z_n\to\Z_{n_1}\times\Z_{n_2}\] be defined by \[f(a)=(x_l,y_l)=\left(\phi_l(a)\mod n_1^l,\phi_l(a)\mod n_2^l\right).\]    It is clear that the sequence $(x_l)$ (respectively, $(y_l)$) satisfies ${}_{n_1}^{}T^l_k(x_l)=x_k$ (respectively, ${}_{n_2}^{}T^l_k(y_l)=y_k$) for any $k\leq l$, and so $(x_l)$ (respectively, $(y_l)$) determines a unique point in $\Z_{n_1}$ (respectively, $\Z_{n_2}$). The fact that $f$ is an isomorphism follows from the Chinese remainder theorem. %Injectively and surjectivity of $f$ follow from that of $f_l$. 
The isomorphism (\ref{eqn:isomorphism}) now follows by repeatedly applying $f$ to distinct pairs of prime factors. 

Moreover, there is an isomorphism $g\colon\Z_{p^j}\to \Z_p$ defined as follows. The number $g(a)$ is $a$ with each coefficient expanded to $a_i=a_{i,j-1} p^{j-1}+\dots+a_{i,1}p+a_{i,0}$, where $a_{i,k}\in\{0,\dots, p-1\}$. Composing $g$ with the isomorphism (\ref{eqn:isomorphism}) shows that, in fact, there is an isomorphism $\Z_n\to\Z_{p_1}\times\cdots\times Z_{p_k}$.

We use the isomorphism $g$ solely to describe the ideals of $\Z_{p^j}$. The non-zero ideals of $\Z_p$ are exactly $p^i\Z_p=\{\sum_{j=i}^\infty a_jp^j\mid a_j\in\{0,\dots, p-1\}\}$. Using the above isomorphism, it is clear that the non-zero ideals of $\Z_{p^j}$ are exactly $g^{-1}(p^i\Z_p)=p^i\Z_{p^j}$. 

We now give a technical description of when elements of $\Z_{p^j}$ and, more generally, $\Z_n$ are contained in a particular ideal.  On a first reading the reader may want to skip Lemmas \ref{lem:pjideal} and \ref{lem:partialsums} and simply read Example \ref{ex:pjideal} and Remark \ref{rem:0ideal} instead.

The following lemma describes when an element $a\in\Z_{p^j}$ is contained in an ideal $p^i\Z_{p^j}$.  By the above discussion, this occurs when the image of the element under the isomorphism $g$ is contained in $g(p^i\Z_{p^j})=p^i\Z_p$.  An element  $x\in \Z_p$ is in the ideal $p^i\Z_p$ exactly when $\phi_i(x)\equiv 0\mod p^{i-1}$.  Since $g$ expands the coefficients of $a$, the definition of $L$ in the statement of the lemma is simply the smallest positive integer such that the expansion of $\phi_L(a)$ contains $\phi_i(g(a))$. Equivalently, it is the smallest positive integer such that the expansion of the  $p^{(L-1)j}$ term in $a$ contains $p^{i-1}$, which is the largest power of $p$ appearing in $\phi_i(g(a))$.  The largest power of $p$ in the expansion of the $p^{(L-1)j}$ term of $a$ is $p^{(L-1)j+j-1}=p^{Lj-1}$, and thus $L$ is the smallest positive integer such that $Lj-1\geq i-1$. 

\begin{lem} \label{lem:pjideal}For any $a\in\Z_{p^j}$, \[ a\in p^i\Z_{p^j} \qquad \iff \qquad \phi_{L}(a)\equiv 0\mod p^{i},\] where $L=\left\lceil\frac{i}{j}\right\rceil$. Moreover, 
\[a\in (0) \qquad \iff \qquad a=0\qquad \iff \qquad \phi_s(a)=0\] for all $s$.
\end{lem}%\begin{equation} \label{eqn:inideal} a\in p^i\Z_{p^j} \qquad \iff \qquad \phi_{l}(a)\equiv \begin{cases}0\mod p^{(l-1)j} & l<L  \\  0\mod 2^i & l=L\end{cases}.\end{equation} \\
\begin{proof}
For the first statement, notice that $a\in p^i\Z_{p^j}$ if and only if $g(a)\in p^i\Z_p$ if and only if $\phi_{i}(g(a))\equiv 0\mod p^i$ if and only if $\phi_{L}(a)\equiv 0\mod p^i$. The first two ``if and only if" statements are clear, while the last  follows from the following calculation:\begin{align*} \phi_L(a)&=a_{L-1}p^{(L-1)j}+\cdots+a_1p^j+a_0 \\
&=a_{L-1,j-1}p^{Lj-1}+\cdots +a_{L-1,k}p^{i-1}+\cdots + a_{L-1,0}p^{(L-1)j} +a_{L-2,j-1}p^{(L-1)j-1} +\cdots +a_{0,1}p+a_{0,0}  \\
&=a_{L-1,j-1}p^{Lj-1}+\cdots +a_{L-1,k+1}p^{i}+\phi_{i}(g(a)),
\end{align*}
where $k\in\{0,\dots, j-1\}$ is such that $(L-1)j+k=i-1$. Now, if $\phi_i(g(a))\equiv 0\mod p^i$, then it is clear that $\phi_L(a)\equiv 0\mod p^i$. Similarly, if $\phi_L(a)\equiv 0\mod p^i$, then since $a_{L-1,j-1}p^{(L-1)j}+\cdots +a_{L-1,k}p^{i}\equiv 0\mod p^i$, we must have $\phi_i(g(a))\equiv 0\mod p^i$.

The second statement is just the definition of the zero ideal.
\end{proof}

\begin{ex} \label{ex:pjideal}
We give explicit descriptions of two ideals in $\Z_{2^3}$.
\begin{enumerate}[(a)]
\item First consider the ideal $2\Z_{2^3}$. Then \[a=\dots a_2a_1a_0\in 2\Z_{2^3}\quad \iff \quad g(a)\in 2\Z_2=\left\{\sum_{i=1}^\infty b_i2^i\,\middle\vert \,b_i\in\{0,\dots, p-1\}\right\}.\] Therefore, the only restriction on the partial sums of $a$ is that \[\phi_1(a)=a_0=a_{0,2}2^2+a_{0,1}2+0,\] i.e., $a_0\equiv 0\mod 2$.

\item Next, consider the ideal $2^{13}\Z_{2^3}$. Then \[a=\dots a_2a_1a_0\in 2^{13}\Z_{2^3} \iff g(a)\in2^{13}\Z_2=\left\{\sum_{i=13}^\infty b_i2^i\,\middle\vert\, b_i\in\{0,\dots, p-1\}\right\}.\] In this case, we must have \[\phi_5(a)=a_42^{12}+a_32^9+a_22^6+a_12^3+a_0=a_{4,2}\cdot 2^{14}+a_{4,1}\cdot 2^{13}+0\cdot 2^{12}+0\cdot 2^{11}+0\cdot 2^{10}+\dots +0\cdot 2+0,\]  i.e., $\phi_5(a)\equiv 0\mod 2^{13}$. Note that this also shows that $\phi_i(a)\equiv 0\mod 2^{3i-1}$ for all $i\leq 4$ (for example, $\phi_2(a)=a_12^3+a_0=0\cdot 2^5+0\cdot 2^4+0\cdot 2^3+0\cdot 2^2+0\cdot 2+0\equiv 0\mod 2^5$). 
\end{enumerate}
\end{ex}

Using the isomorphism in (\ref{eqn:isomorphism}), any ideal $\frak a$ of $\Z_n$ can be written as \[\frak a=\frak a_1\times\cdots \times \frak a_k,\] where $\frak a_i=p_i^{a_i}\Z_{p_i^{n_i}}$ for some $a_i$ or $\frak a_i=(0)$.   The next lemma gives a precise description of when an element  $a\in\Z_n$ is contained in an ideal $\frak a=\frak a_1\times\cdots\times\frak a_k$.  The conditions are similar to those in Lemma \ref{lem:pjideal}.  For each $i$ such that $\frak a_i=p_i^{a_i}\Z_{p_i^{n_i}}$, we have a constant $L_i=\lceil \frac{a_i}{n_i}\rceil$ as in Lemma \ref{lem:pjideal}, and one needs only check a condition on a \emph{single} partial sum, namely that $\phi_{L_i}(a)\equiv 0\mod p_i^{a_i}$.  On the other hand, for each $i$ such that $\frak a_i=(0)$, one needs to check that \emph{every} partial sum of $a$ satisfies an appropriate condition.

\begin{lem} \label{lem:partialsums}
Let $\frak a =\frak a_1\times\cdots\times \frak a_k$, where for each $i$, $\frak a_i=p_i^{a_i}\Z_{p_i^{n_i}}$ or $\frak a_i=(0)$. For each $i$, let $L_i=\left\lceil \frac{a_i}{n_i}\right\rceil$. Then for any $a=\dots a_3a_2a_1a_0\in\Z_n$, 
\begin{align*} 
a\in\frak a \quad \iff \quad & \boldsymbol{\cdot} \phi_s(a)\equiv 0\mod p_{i}^{sn_i} \quad \textrm{for all $s$ and all $i$ such that }\frak a_i=(0); and 
\\ & \boldsymbol{\cdot} \phi_s(a)\equiv 0\mod p_i^{a_i}\quad \textrm{for any $s$ such that } s=L_i \textrm{ for some $i$}.
\end{align*}
\end{lem}

\begin{proof}
For any $i$ such that $\frak a_i=p_i^{a_i}\Z_{p_i^{n_i}}$, this follows immediately from the definition of the isomorphism $\Z_n\to  \Z_{p_1^{n_1}}\times\cdots\times \Z_{p_k^{n_k}}$ and Lemma \ref{lem:pjideal}. Fix $i$ such that $\frak a_i=(0)$. Then by the isomorphism in (\ref{eqn:isomorphism}), we have that any $a\in\Z_n$ can be written as $a=(a_1,\dots, a_k)$ where, considering $\Z_n$ and $\Z_{p_i^{n_i}}$ as $\varprojlim \Z/n^l\Z$ and $\varprojlim \Z/p_i^{ln_i}\Z$, respectively, each $a_i$ is given by the sequence $(\phi_s(a)\mod p_i^{sn_i})_{s=1}^\infty$. Now, $a_i\in(0)\subset \Z_{p_i^{n_i}}$ if and only if $g(a_i)\in (0)\subset \Z_{p_i}$.  Recall that $g(a_i)$ is $a_i$ with each coefficient expanded so that $\phi_s(a)\mod p_i^{sn_i}\equiv a_{s-1,n_i-1}p_i^{sn_i-1}+a_{s-1,n_i-2}p_i^{sn_i-2}+\cdots + a_{0,1}p_i + a_{0,0} \mod p_i^{sn_i}$.  From this we see that $g(a_i)\in (0)\subset \Z_{p_i}$ if and only if $\phi_s(a) \equiv 0\mod p_i^{k}$ for all $1\leq k\leq sn_i$ if and only if $\phi_s(a) \equiv 0 \mod p_i^{sn_i}$ for all $s$.
%, which is the case if and only if $\phi_s(a)\mod p_i^{sn_i}\equiv %0\mod p_i^{(s+1)n_i}$ if and only if $\phi_s(a)\equiv 0\mod p_i^{s+1}$. The result follows.
\end{proof}

\begin{rem} \label{rem:0ideal} We point out one particular case of this lemma which will be important in later sections: if some $\frak a_i=(0)$, then $\phi_1(a)=a_0\equiv 0\mod p_i^{n_i}$. 
\end{rem}

We now describe a partial order  on the set of ideals of $\Z_n$.
\begin{defn} \label{def:equivideal}
Define a relation $\leq$ on ideals of $\Z_n$ by $\mathfrak a\leq\mathfrak b$ if $n^k\mathfrak a\subseteq\mathfrak b$ for some $k$. Define an equivalence $\mathfrak a\sim\mathfrak b$ if $\mathfrak a\leq \mathfrak b$ and $\mathfrak b\leq \mathfrak a$.
\end{defn}

\begin{defn} \label{def:fullideal}
An ideal $\mathfrak a=\mathfrak a_1\times\cdots \times\mathfrak a_k$ is \emph{full} if $\mathfrak a_j$ is either $(0)$ or $\Z_{p_j^{n_j}}$ for every $j=1,\dots, k$.
\end{defn}

%\begin{rem}\label{rem:isoidealstosubsets}Let $\mathfrak A$ be the poset of full ideals of $\Z_n$, with the partial order given by inclusion. Then there is an isomorphism of posets $\mathfrak A\to 2^{\{1,\dots,k\}}$, where $2^{\{1,\dots,k\}}$ also has the partial order given by inclusion, defined by $\mathfrak a\mapsto \{j\in \{1,\dots, k\}\mid \mathfrak a_j\cong \Z_{p_j}\}$.
%\end{rem}

\begin{lem} \label{lem:equivalentfull}
For any $n$, there is a unique full ideal in each equivalence class of ideals of $\Z_n$.
\end{lem}

\begin{proof}
Let $\frak a=\frak a_1\times\cdots\times \frak a_k$ be an ideal of $\Z_n$, and consider the full ideal $\frak b=\frak b_1\times\cdots\times \frak b_k$ where $\frak b_i=(0)$ if and only if $\frak a_i=(0)$. Recall that if $\frak b_i\neq (0)$, then $\frak b_i=\Z_{p_i^{n_i}}$. We claim that $\frak a\sim \frak b$. It is clear that $\frak a \subseteq \frak b$, and thus $\frak a\leq \frak b$. For any $i$ such that $\frak a_i\neq (0)$, let $\frak a_i=p^{a_i}\Z_{p_i^{n_i}}$, and let $A=\max_i\{a_i\}$. We will show that $n^A\frak b\subset \frak a$. 
We have \[n^A\frak b =n^A\frak b_1\times \cdots \times n^A\frak b_k,\]
 and, for each $i$, \[n^A\frak b_i=p_1^{An_1}\cdots p_k^{An_k}\Z_{p_i^{n_i}}.\] For all $j\neq i$, the element $p_j^{An_j}$ is a unit in $\Z_{p_i^{n_i}}$, and thus $p_j^{n_j}\Z_{p_i^{n_i}}=\Z_{p_i^{n_i}}$. Therefore, we have \[n^A\frak b_i=p_i^{An_i}\Z_{p_i^{n_i}},\] and since $A\geq a_i$ by definition, it follows that \[n^A\frak b_i=p_i^{An_i}\Z_{p_i^{n_i}}\subseteq p_i^{a_i}\Z_{p_i^{n_i}}=\frak a_i.\]  Consequently, \[n^A\frak b\subseteq\frak a,\] which implies that $\frak b\leq \frak a$. Therefore, $\frak a\sim\frak b$.
 
Suppose next that there are two distinct full ideals, $\frak b=\frak b_1\times\cdots\times \frak b_k$ and $\frak c=\frak c_1\times\cdots\times\frak c_k$ with $\frak b\sim\frak a\sim\frak c$. Then $\frak b\sim \frak c$, which immediately implies that $\frak b_i=(0)$ if and only if $\frak c_i=(0)$. Indeed, if there is an $i$ such that (without loss of generality) $\frak b_i=(0)$ but $\frak c_i\neq(0)$, then there is no power $C$ of $n$ such that $p_i^C\frak c_i\subset \frak b_i$, and so $\frak b\not\sim\frak c$, which is a contradiction. Thus by the definition of full ideals, we conclude that $\frak b=\frak c$.
\end{proof}

\medskip
While we mostly work with $n$--adic integers in this paper, we will occasionally need the notion of an $n$--adic number, as well. 

\begin{defn}
An \emph{$n$--adic number} is an infinite series 
\[
\sum_{i=m}^\infty a_i n^i
\]
where $a_i\in \{0,1,\dots, n-1\}$ and $m\in \Z$ can be positive, negative, or zero.  If $m\geq 0$, then such an element is an $n$--adic integer.  If $m=-\ell$ for $\ell\in\Z_{>  0}$, then such an element can also be written in its base $n$ expansion as
 \[
\ldots a_3a_2a_1a_0.a_{-1}\ldots a_{-\ell}.
\]
We denote the set of $n$--adic numbers by $\Q_n$.
\end{defn}
  Letting $S=\{n,n^2,n^3,\dots\}$, we see that $\Q_n=S^{-1}\Z_n$ is the localization of $\Z_n$ at $S$.  In particular, $\Q_n$ is a ring.  If $n=p^k$ is a power of a prime, then $\Z_n$ is an integral domain and $\Q_n$ is its field of fractions.  If $n$ is not a power of a prime, however, then $\Z_n$ is not an integral domain and $\Q_n$ will not be a field.  %If  $n=p_1^{n_1}p_2^{n_2}\cdots p_l^{n_k}$, then it follows from the definition of $\Q_n$ and the isomorphism $\Z_n\cong \Z_{p_1}\times \cdots \times \Z_{p_k}$ that  then $\Q_n\cong \Q_{p_1}\times\cdots\times \Q_{p_k}$.

The only property of $n$--adic numbers we will need is that one can define an $n$--adic absolute value on $\Q_n$. 

\begin{defn}\label{def:absval}
Given $q=\sum_{i=m}^\infty a_in^i\in \Q_n$ with $a_m\neq 0$, we define the \emph{$n$--adic absolute value} of $q$ to be
\[
\|q\|_n=n^{-m}.
\]
\end{defn}
In particular, $q=\sum_{i=m}^\infty a_in^i\in \Z_n$ if and only if $m\geq 0$ if and only if $\|q\|_n\leq 1$.

\subsection{$BS(1,n)$}
Fix $n=p_1^{n_1}p_2^{n_2}\dots p_k^{n_k}$, and recall that $BS(1,n)=\langle a,t\mid tat^{-1}=a^n\rangle$. Let $\tau: BS(1,n)\to \Z$ be the homomorphism $a\mapsto 0, t\mapsto 1$. Then there is a short exact sequence \[0\to H\to BS(1,n)\xrightarrow{\tau}\Z\to 0,\] where $H:=\ker(\tau)\cong \Z\left[\frac1n\right]$. This gives rise to an isomorphism $BS(1,n)\cong \Z\left[\frac{1}{n}\right]  \rtimes_\alpha \mathbb Z$, where $\alpha(x)=n\cdot x$ for $x\in \Z\left[\frac{1}{n}\right]$. For the rest of this paper we will make the identifications \[H=\Z\left[\frac{1}{n}\right],\] and \[BS(1,n)=H\rtimes_\alpha \Z=\langle H, t \mid txt^{-1}=\alpha(x) \text{ for } x\in H\rangle.\]

In addition to the standard representation of elements of $H$ as Laurent polynomials in $n$, we also represent elements by their $n$--ary expansion; e.g. $\frac1n=0.1$ while $n+\frac1n+\frac{1}{n^4}=10.1001$. We switch between these representations interchangeably.

Given an element $x=\pm x_kx_{k-1}\cdots x_2x_1x_0.x_{-1}x_{-2}\cdots x_{-m}\in H$, we have $0\leq x_i < n$ for all $-m\leq i\leq k$. We call $-m$ the \textit{leading negative place of $x$}, which we denote by \[p(x)=-m.\] We call $x_{-m}$ the \emph{leading negative term of $x$}, which we denote by \[c(x)=x_{-m}.\]  The automorphism $\alpha$ acts on $H$ by multiplication by $n$, which has the effect of adding one to each index, so that the $i$--th term of the image of $x$ is the $(i+1)$--st term of $\alpha(x)$. For example, $\alpha(21.021311)=210.21311$ (here we are assuming that $n\geq 4$).

\begin{lem}\label{lem:movegenerators}
Let $K$ be a group of the form $K=K'\rtimes_\alpha \Z=\langle K',s\mid sxs^{-1}=\alpha(x) \textrm{ for } x\in K'\rangle$.  Suppose $Q\subseteq K'$ is a subset so that $Q\cup \{s^{\pm 1}\}$ is a generating set of $K$ and $\alpha(Q)\subset Q$.  Then any element $w\in K$ can be written as \[w=s^{-r} x_1\ldots x_m s^\ell\] where $r,\ell\geq 0$, $x_i\in Q$ for all $i$, and $r+\ell+m= \|w\|_{Q\cup\{s^{\pm 1}\}}$. Moreover, if $w\in K'$ then $r=\ell$.
\end{lem}

\begin{proof}
Write $w$ as a reduced word in $Q\cup \{s^{\pm 1}\}$. By using the relations $sx=\alpha(x)s$ and $xs^{-1}=s^{-1}\alpha(x)$ for $x\in Q$, we may move all copies of $s$ in $w$ to the right and all copies of $s^{-1}$ to the left without increasing the word length of $w$ in $Q\cup \{s^{\pm 1}\}$. The result is an expression of $w$ as a reduced word, \[w=s^{-r} x_1\ldots x_m s^\ell\] where $r,\ell\geq 0$ and $x_i\in Q$ for all $i$.  Since the word length of $w$ has not changed, we have $r+\ell+m= \|w\|_{Q\cup\{s^{\pm 1}\}}$. The second statement is clear.
\end{proof}

%%%%%%%%%%%%%%%%%%%%%%%%%%%%%%
\section{Confining subsets of $H$}

We first describe two particular subsets of $H= \Z\left[\frac1n\right]$ which are strictly confining under the action of $\alpha$ and $\alpha^{-1}$, respectively.

\begin{lem}
The subset \begin{equation} \label{eqn:Q+} Q^+=\{x\in H\mid x= \pm x_kx_{k-1}\cdots x_2x_1x_0 \textrm{ for some } k\in \mathbb N\}= \mathbb Z\subset H\end{equation} is strictly confining under the action of $\alpha$. The subset \begin{equation} \label{eqn:Q-} Q^-=\{x\in H\mid x=\pm 0.x_{-1}x_{-2}\cdots x_{-m} \textrm{ for some } m\in\mathbb N\}\subset H\end{equation} is strictly confining under the action of $\alpha^{-1}$.
\end{lem}

\begin{proof}
We will verify that Definition \ref{def:confining} holds for $Q^-$; the proof for $Q^+$ is analogous. We have \[\alpha^{-1}(Q^-)=\{x\in H\mid x= \pm 0.0x_{-1}x_{-2} \cdots x_{-m} \mid \textrm{ for some } m\in \mathbb{Z}_{>0}\}.\] Thus $\alpha^{-1}(Q^-)\subset Q^-$, so (a) holds. Moreover, it is clear that $\bigcup_{n\geq 0}\alpha^n(Q^-)=H$. Indeed, let $x=\pm x_kx_{k-1}\cdots x_0.x_{-1}x_{-2}\cdots x_{-m}$ be any element of $H$. Since $\alpha^{-(k+1)}(x)\in Q^-$, it follows that $x\in \alpha^{k+1}(Q^-)$, and thus (b) holds. Finally, let $x,y\in Q^-$. Then we have \[x+y=\pm z_0.z_{-1} \cdots z_{-m}\] where each $z_i \in \{0,\ldots,n-1\}$. Hence $\alpha^{-1}(x+y)\in Q^-$ and (c) holds with $k_0=1$.  Thus $Q^-$ is confining under the action of $\alpha^{-1}$.  To see that $Q^-$ is \emph{strictly} confining, note that $0.1\in Q \setminus \alpha^{-1}(Q)$.
\end{proof}

The following lemma appears as \cite[Lemma~4.9 \& Corollary~4.10]{Bal}; we include a proof here for completeness. Recall that given two (possibly infinite) generating sets $S,T$ of a group $G$, we say $[S]=[T]$ if $\sup_{g\in S}\|g\|_T<\infty$ and $\sup_{h\in T}\|h\|_S<\infty$.

\begin{lem}\label{lem:elements}  Suppose $Q$ is a symmetric subset of $H$ which is confining under the action of $\alpha$.
Let $S$ be a symmetric subset of $H$ such that there exists $K\in \Z_{\geq 0}$ with $\alpha^K(g)\in Q$ for all $g\in S$. Then \[\overline{Q}=Q\cup \bigcup_{i\geq 0} \alpha^i(S)\] is confining under the action of $\alpha$ and \[ [Q\cup \{t^{\pm 1}\}]=[\overline{Q}\cup \{t^{\pm 1}\}].\]
\end{lem}

We note that this lemma applies, for example, to all finite symmetric subsets $S$ of $H$.

\begin{proof}[Proof of Lemma \ref{lem:elements}]
First we prove that $\overline{Q}$ is confining under the action of $\alpha$. 

Conditions (a) and (b) from Definition \ref{def:confining} are clear (using that $Q\subseteq \overline{Q}$ for condition (b)).

To see that condition (c) holds, note that for any $i\geq 0$, and any $g\in S$, \[\alpha^K(\alpha^i(g))=\alpha^i(\alpha^K(g))\in\alpha^i(Q)\subseteq Q.\] We also have $\alpha^K(g)\in Q$ for any $g\in\overline{Q}$. Hence, if $g,h\in \overline{Q}$ we have $\alpha^K(g)\in Q$ and $\alpha^K(h)\in Q$, and therefore \[\alpha^{K+k_0}(g+h)=\alpha^{k_0}(\alpha^K(g)+\alpha^K(h))\in \alpha^{k_0}(Q + Q)\subseteq Q \subseteq \overline{Q},\] where $k_0$ is large enough so that $\alpha^{k_0}(Q+Q)\subseteq Q$.  Therefore (c) holds with constant $K+k_0$.

To see that $[Q\cup \{t^{\pm 1}\}]=[\overline{Q}\cup \{t^{\pm 1}\}]$, note first of all that clearly $[\overline{Q}\cup \{t^{\pm 1}\}]\preccurlyeq [Q \cup \{t^{\pm 1}\}]$. On the other hand, by our above observation, $\overline{Q}$ is really just a finite union: \[\overline{Q}=Q\cup \bigcup_{i=0}^{K-1} \alpha^i(S).\] For each $i$ between 0 and $K-1$ and each $g\in S$, we have \[\alpha^i(g)=\alpha^{-(K-i)}(\alpha^K(g))=t^{-(K-i)}\alpha^K(g)t^{(K-i)}\] and $\alpha^K(g)\in Q$. Hence  \[\| \alpha^i(g) \|_{Q\cup \{t^{\pm 1}\} } \leq 2(K-i)+1\leq 2K+1.\] In other words, any element of $\overline{Q}\cup \{t^{\pm 1}\}$ has word length at most $2K+1$ with respect to $Q\cup \{t^{\pm 1}\}$, so $[Q \cup \{t^{\pm 1}\}] \preccurlyeq [\overline{Q} \cup \{t^{\pm 1}\}]$.
\end{proof}

\begin{lem} \label{lem:Q+subset}
For any $Q\subseteq H$ which is confining under the action of $\alpha$, we have $[Q\cup\{t^{\pm 1}\}] \preccurlyeq [Q^+ \cup \{t^{\pm 1}\}]$.
\end{lem}

\begin{proof}
We show that every element of $Q^+=\Z$ has bounded word length with respect to $Q\cup \{t^{\pm 1}\}$. First, we apply Lemma \ref{lem:elements} with $S=\{\pm 1\}$ to pass to $\overline{Q}\supset Q$ such that $\{\pm 1\} \subseteq \overline{Q}$ and $[Q\cup \{t^{\pm 1}\}]=[\overline{Q} \cup \{t^{\pm 1}\}]$.

We begin by showing that every element of $\Z_{>0}=\{1,2,\ldots\}$ has bounded word length with respect to $\overline{Q}\cup\{t^{\pm 1}\}$. Choose $k_1$ such that $\alpha^{k_1}(\overline{Q} + \overline{Q})\subseteq \overline{Q}$. We actually initially prove that every element of $\alpha^{k_1}(\Z_{>0})=\{n^{k_1},2n^{k_1},3n^{k_1},\ldots\}$ has bounded word length with respect to $\overline{Q}\cup\{t^{\pm 1}\}$. If every such element has word length at most $ L$ then every element of $\Z_{>0}$ has word length less than $ L+n^{k_1}$ with respect to $\overline{Q}\cup\{t^{\pm 1}\}$ because such an element can be written as \[a n^{k_1}+\underbrace{1+\cdots +1}_{< n^{k_1} \text{ times}}, \quad \textrm{where} \quad a\in \Z_{>0}\] and $1\in \overline{Q}$.

The proof of this weaker statement is by induction. First, note that $\alpha^{k_1}(1)=n^{k_1}\in \overline{Q}$. Hence every element of the set \[\{n^{k_1},2n^{k_1}, 3n^{k_1}, \ldots,n^{2k_1}=n^{k_1}\cdot n^{k_1}\}\] has word length at most $ n^{k_1}$ with respect to $\overline{Q}\cup\{t^{\pm 1}\}$. Suppose for induction that every element of \[\{n^{(l-1)k_1},n^{(l-1)k_1}+n^{k_1},n^{(l-1)k_1}+2n^{k_1},\ldots,n^{lk_1}\}=\{an^{k_1}\mid a\in \Z_{>0}\}\cap [n^{(l-1)k_1},n^{lk_1}]\] has word length at most $ n^{k_1}$ with respect to $\overline{Q}\cup\{t^{\pm 1}\}$. Enumerate the elements of this set as \[x_0=n^{(l-1)k_1},x_1=n^{(l-1)k_1}+n^{k_1},\ldots, x_s=n^{lk_1}.\] Consider an element \[y\in \{n^{lk_1},n^{lk_1}+n^{k_1},n^{lk_1}+2n^{k_1},\ldots,n^{(l+1)k_1}\}=\{an^{k_1}\mid a\in \Z_{>0}\}\cap [n^{lk_1},n^{(l+1)k_1}].\] Such an element $y$ satisfies \[n^{k_1}x_j \leq y\leq n^{k_1}x_{j+1}=n^{k_1}(x_j+n^{k_1})=n^{k_1}x_j + n^{2k_1}\] for some $j$. Hence we have \begin{equation} \label{eqn:mid y} y=n^{k_1}x_j+an^{k_1},\end{equation} where $0\leq a\leq n^{k_1}$. Since $x_j$ has word length at most $ n^{k_1}$, we may write $x_j=g_1+\cdots+g_{n^{k_1}}$ where all $g_i\in \overline{Q}$ and $g_i=0$ for all $i>\|x_j\|_{\overline Q\cup\{t^{\pm 1}\}}$. Thus we can rewrite  equation (\ref{eqn:mid y}) as \[\begin{tabular}{l l l} 

$y$ & $=$ & $n^{k_1}(g_1+\ldots+g_{n^{k_1}})+n^{k_1}(\underbrace{1+\cdots+1}_{a\leq n^{k_1} \text{ times}})$ \\
& $=$ & $n^{k_1}(g_1+1)+n^{k_1}(g_2+1)+\cdots +n^{k_1}(g_a+1)+n^{k_1}(g_{a+1})+\cdots+n^{k_1}(g_{n^{k_1}})$. \end{tabular}\] In this last sum, every term is an element of $\overline{Q}$, and there are $n^{k_1}$ terms. Thus $\|y\|_{\overline{Q}\cup \{t^{\pm 1}\}}\leq n^{k_1}$. This completes the induction.

We have shown so far that every element of $\Z_{>0}$ has bounded word length with respect to $\overline{Q}$. A completely analogous argument using multiples of $-n^{k_1}$ proves that every element of $\Z_{<0}=\{-1,-2,\ldots\}$ has bounded word length with respect to $\overline{Q}$. Hence we have shown \[[Q\cup \{t^{\pm 1}\}]=[\overline{Q}\cup \{t^{\pm 1}\}]\preccurlyeq [\Z\cup \{t^{\pm 1}\}] = [Q^+\cup \{t^{\pm 1}\}]. \qedhere\]
\end{proof}

\subsection{Subsets confining under the action of $\alpha$ and ideals of $\Z_n$}
In this subsection, we describe the connections between subsets of $H$ which are confining under the action of $\alpha$ and ideals of $\Z_n$. 

\subsubsection{From confining subsets to ideals} \label{section:conftoideal}
We begin by describing a way to associate an ideal of $\Z_n$ to a symmetric subset $Q$ of $H$ which is confining under the action of $\alpha$. We define

\begin{equation}\label{eqn:idealL}
\mathcal{I}(Q)=\Set{
\ldots x_2 x_1 x_0 \in \Z_n \,\Bigg\vert \,\begin{aligned} & \text{for any } t\geq 0, \exists a \in Q \text{ with } a=a_r \cdots a_0. x_t \cdots x_0 \\ &\text{ for some }a_r,\dots, a_0\in\{0,\dots, n-1\}\end{aligned}
}.
\end{equation} That is, an element $\ldots x_2 x_1 x_0$ is in $\mc I(Q)$ if for any $t$, there exists a positive element of $Q$ whose fractional part is $0.x_t\cdots x_0$.  Note in particular that $\mathcal{I}(Q)$ is nonempty for any $Q$ as above. To see this, first notice that $Q$ always contains a positive \textit{integer} $a=a_r\cdots a_0$. We may equivalently write \[a=a_r\cdots a_0. \underbrace{0 \cdots \cdots 0}_{t \text{ times}}\] and since $t$ is arbitrary, this shows that $0 \ (= \ldots 0 0 0 ) \in \mathcal{I}(Q)$.

\begin{lem} \label{lem:closed}
The set $\mathcal{I}(Q)\subseteq \Z_n$ is closed.
\end{lem}

\begin{proof}
Let $x\in \overline{\mathcal{I}(Q)}$ and write $x=\ldots x_2 x_1 x_0$. Then for any $t\geq 0$, there exists $y\in \mathcal{I}(Q)$ with $y=\ldots y_2 y_1 y_0$ and $y_i=x_i$ for $i\leq t$. By definition of $\mathcal{I}(Q)$ there exists $a\in Q$ with $a=a_r\cdots a_0.y_t \cdots y_0$. But then of course we also have $a=a_r\cdots a_0. x_t \cdots x_0$. Since $t$ is arbitrary, this implies that $x\in \mathcal{I}(Q)$.
\end{proof}

\begin{lem}
In the notation above, $\mathcal{I}(Q)$ is an ideal of $\Z_n$.
\end{lem}

\begin{proof}
First we show that $\mathcal{I}(Q)$ is closed under addition. Let \[x=\dots x_2 x_1 x_0 \text{ and } y=\dots y_1 y_1 y_0 \in \mathcal{I}(Q)\] and\[ \begin{array}{l l l l l}
& \cdots & x_2 & x_1 & x_0 \\
+ & \cdots & y_2 & y_1 & y_0 \\
\hline
 & \cdots & z_2 & z_1 & z_0\\
\end{array}\] Let $k_0$ be large enough that $\alpha^{k_0}(Q+Q)\subseteq Q$. By definition of $\mathcal{I}(Q)$, for any $t\geq 0$ there exist (positive numbers) \[ \begin{array}{l}
a=a_r\cdots a_0. x_{t+k_0} x_{t+k_0-1} \cdots x_0 \\
b=b_s\cdots b_0. y_{t+k_0} y_{t+k_0-1} \cdots y_0 \end{array}\] in $Q$. We see immediately that $a+b$ is given by \[c_u\cdots c_0. z_{t+k_0} z_{t+k_0-1} \cdots z_0\] for some $c_u, \ldots, c_0 \in \{0,\ldots, n-1\}$. This implies that \[\alpha^{k_0}(a+b)=c_u\cdots c_0 z_{t+k_0} z_{t+k_0-1}\cdots z_{t+1}. z_t \cdots z_0\in Q.\] Since $t$ is arbitrary, this implies that $z\in \mathcal{I}(Q)$.

Now we show that $\mathcal{I}(Q)$ is closed under multiplication by elements of $\Z_n$. Let \[x \in \mathcal{I}(Q) \text{ and } p=\ldots p_2 p_1 p_0 \in \Z_n.\] For every $t\geq 0$ we have \[p_t \ldots p_1 p_0 \cdot x = \underbrace{x+ \cdots + x}_{p_t \ldots p_1 p_0 \text{ times}}\in \mathcal{I}(Q)\] by the above paragraph. Note that $p\cdot x$ is the limit of the sequence $\{p_t\ldots p_0 \cdot x\}_{t=0}^\infty\subseteq \mathcal{I}(Q)$. But by Lemma \ref{lem:closed}, $\mathcal{I}(Q)$ is closed, so this implies that $p\cdot x\in \mathcal{I }(Q)$ as well.
\end{proof}

\subsubsection{From ideals to confining subsets}

We next describe how to associate a subset of $H$ which is confining under the action of $\alpha$ to an ideal of $\Z_n$. For any ideal $\mathfrak b$ of $\Z_n$, let 

\begin{equation}\label{eqn:subsetS}
\mc C(\frak b)= \Set{
(-1)^\delta x_r\cdots x_0. x_{-1} \cdots x_{-s}\in H \,\Bigg\vert \,\begin{aligned} &  \delta \in\{0,1\}  \text{ and } 
 \exists b\in \mathfrak{b}  \text{ with }  b= \ldots  c_2 c_1 x_{-1} \ldots x_{-s} \\&\text{ for some } c_1, c_2, \ldots \in \{0,\ldots, n-1\}\end{aligned} 
}.
\end{equation}
Thus $\mc C(\mathfrak b)$ is the set of elements of $H$ whose fractional parts appear as the tail end of digits of some element of the ideal $\frak b$. 

\begin{rem}  \label{rem:Zinsubset} Since $0\in\frak b$ for any ideal $\frak b$, it follows that $\mc C(\frak b)$ must contain $\Z$.
\end{rem}

\begin{lem} \label{lem:idealtoconfining}
In the notation above, $\mathcal{C}(\mathfrak b)$ is confining under the action of $\alpha$.
\end{lem}

\begin{proof}
We will check the conditions of Definition \ref{def:confining}. 

Let \[x=x_s\cdots x_1x_0.x_{-1}x_{-2}\cdots x_{p(x)}\in\mathcal C(\frak b).\]   By definition of $\mathcal C(\frak b)$, there is an element \[b=\dots c_2 c_1x_{-1}x_{-2}\dots x_{p(x)}\in\frak b,\]  and so \[\alpha(x)=x_s\cdots x_0x_{-1}.x_{-2}\cdots x_{p(x)}\in\mathcal C(\frak b).\]  Thus $\alpha(\mathcal C(\frak b))\subseteq \mathcal C(\frak b)$, and Definition \ref{def:confining}(a) holds. 

Since $\Z\subseteq \mathcal C(\frak b)$ by Remark \ref{rem:Zinsubset}, we  have $\bigcup_{i=0}^\infty \alpha^{-i}(\mathcal C(\frak b))=H$, and so Definition \ref{def:confining}(b) holds.

Let $x,y\in\mathcal C(\frak b)$. We first deal with the case that $x$ and $y$ are both positive. We want to show that $x+y\in\mathcal C(\frak b)$. Let $x=x_r\cdots x_0.x_{-1}\cdots x_{p(x)},$ and $y=y_s\cdots y_0.y_{-1}\cdots y_{p(y)}$.  By adding initial zeros, we may take $r=s$.  We also assume without loss of generality that $p(x)\leq p(y)$. Then $x+y=z$, where $z$ is given by
\[\begin{array}{ccccccccccccccccc}
&&x_r &\cdots &x_0 & .x_{-1}& x_{-2} & \cdots & x_{p(y)} &\cdots & \cdots &\cdots & x_{p(x)} \\
+&& y_r&\cdots&y_0&  .y_{-1} & y_{-2} & \cdots & y_{p(y)} & 0&\cdots &0 &0\\ \hline
&z_t&\cdots&\cdots&z_0 &.z_{-1}&z_{-2}&\cdots &z_{p(y)} &\cdots &\cdots&\cdots  &z_{p(x)}
\end{array}
\]
where here we've assumed without loss of generality that $r\geq s$ (the same argument works if $r<s$).

By the definition of $\mathcal C(\frak b)$, there exist $a,b\in\frak b$ with $a=\dots x_{-1}x_{-2}\dots x_{p(x)}$ and $b=\dots y_{-1}y_{-2}\dots y_{p(y)}$. Since $\frak b$ is an ideal, \[n^{p(y)-p(x)}b=\dots y_{-1}y_{-2}\dots y_{p(y)}\underbrace{0\dots\dots\dots\dots 0}_{p(y)-p(x) \textrm{ times}}\in \frak b\] and  \[a+n^{p(y)-p(x)}b\in \frak b,\] where $a+n^{p(y)-p(x)}b$ is given by 
\[\begin{array}{cccccccccc}
&\dots & x_{-1}& x_{-2} & \dots & x_{p(y)} & \dots & \dots & \dots & x_{p(x)} \\
+&\dots &  y_{-1} & y_{-2} & \dots & y_{p(y)} & 0&\dots &\dots &0 \\ \hline
&\dots &z_{-1}&z_{-2}&\dots &z_{p(y)} &\dots &\dots&\dots  &z_{p(x)}.
\end{array}
\] 
Therefore, $z=x+y\in \mathcal C(\frak b)$ by the definition of $\mc C(\frak b)$.

If $x,y \in \mathcal C(\frak b)$ are both negative, then we show in a completely analogous way that $x+y\in \mathcal C(\frak b)$.

We now consider the case that one of $x,y$ is positive and the other is negative. By possibly multiplying $x+y$ by $-1$, we assume without loss of generality that $x=x_r\cdots x_0 . x_{-1} \cdots x_{p(x)}$ and $y=-y_s \cdots y_0. y_{-1}\cdots y_{p(y)}$ with $x\geq |y|$ so that also $r\geq s$. Then $x+y=z$, where $z$ is given by \[\begin{array}{cccccccccccccccc}  & x_r & \cdots &x_t & \cdots & x_s &\cdots & x_0 & .x_{-1} & x_{-2} & \cdots & x_{p(y)} & \cdots & \cdots & x_{p(x)} \\
- & & &  & & y_s & \cdots & y_0 & .y_{-1} & y_{-2} & \cdots & y_{p(y)} & 0 & \cdots & 0 \\
\hline
& & & z_t & \cdots & z_s & \cdots & z_0 & .z_{-1} & z_{-2} & \cdots & z_{p(y)} & \cdots  & \cdots & z_{p(x)}.\end{array}\] (Here we are assuming that $p(x)\leq p(y)$ but the argument is easily modified if $p(x)>p(y)$). By definition of $\mathcal C(\frak b)$, there exist elements $c,d\in \frak b$ with $c=\ldots c_2 c_1 x_{-1} \ldots x_{p(x)}$ and $d=\ldots d_2 d_1 y_{-1} \ldots y_{p(y)}$. Then $c-n^{p(y)-p(x)}d\in \frak b$  is given by \[\begin{array}{ccccccccc}

& \ldots & c_1 & x_{-1} & \ldots & x_{p(y)} & \ldots & \ldots  &x_{p(x)} \\
-& \ldots &d_1 & y_{-1} & \ldots & y_{p(y)} & 0 & \ldots & 0 \\
\hline
& \ldots & \ldots & z_{-1} & \ldots & z_{p(y)} & \ldots & \ldots & z_{p(x)}. \\ \end{array}\] Hence we see that $z\in \mathcal C (\frak b)$, as desired.

By the above discussion,  Definition \ref{def:confining}(c) holds with $k_0=0$. We conclude that $\mathcal C(\frak b)$ is confining under the action of $\alpha$.
\end{proof}

\begin{lem} \label{lem:Ziterate}
Let $Q\subseteq H$ be confining under the action of $\alpha$. Then there exists $K>0$ such that $\alpha^K(\Z)\subseteq Q$.
\end{lem}

\begin{proof}
By Lemma \ref{lem:Q+subset}, every element of $\Z=Q^+$ has uniformly bounded word length with respect to the generating set $Q\cup \{t^{\pm 1}\}$ of $H$. Consider an element $w\in \Z$. By Lemma \ref{lem:movegenerators} we may write $w$ as a reduced word \[w=t^{-r} x_1\ldots x_m t^r,\] where $r\geq 0$ and $x_i\in Q$ for all $i$. This gives us \[w=\alpha^{-r}(x_1)+\cdots +\alpha^{-r}(x_m).\]  Since $\|w\|_{Q\cup\{t^{\pm 1}\}}$ is uniformly bounded, we have both $r$ and $m$ are uniformly bounded, say $r,m\leq R$. Hence we have \[\alpha^R(w)=\alpha^{R-r}(x_1)+\cdots + \alpha^{R-r}(x_m),\] and $\alpha^{R-r}(x_i)\in Q$ for each $i$. Thus, $\alpha^R(w)\in Q^R$, where $Q^R$ represents the words of length at most $ R$ in $Q$. Consequently, $\alpha^{Rk_0}(\alpha^R(w))\in \alpha^{Rk_0}(Q^R)\subseteq Q$ where the last inclusion follows by Definition \ref{def:confining}(c). Thus we see that $\alpha^K(\Z)\subseteq Q$, where $K=Rk_0+R$.
\end{proof}

\begin{lem} \label{lem:SLTsubsetT}
Let $Q\subseteq H$ be confining under the action of $\alpha$. Then there exists $M>0$ such that $\mathcal{C}(\mathcal{I}(Q))\subseteq \alpha^{-M}(Q)$.
\end{lem}

\begin{proof}
Let $a\in \mathcal{C}(\mathcal{I}(Q))$. Since $Q$ and $\mathcal{C}(\mathcal{I}(Q))$ are symmetric, we may suppose that $a=a_r \cdots a_0.a_{-1} \cdots a_{-s}$ is positive. By definition of $\mathcal{C}(\mathcal{I}(Q))$ there exists an element \[x=\ldots x_2 x_1 a_{-1} \ldots a_{-s} \in \mathcal{I}(Q).\] Then by definition of $\mathcal{I}(Q)$, there exists an element \[b=b_t \cdots b_0. a_{-1} \cdots a_{-s}\in Q.\] 
We may add an integer $c$ to $b$ to obtain \[c+b=a_r\cdots a_0.a_{-1} \cdots a_{-s}=a,\] and by Lemma \ref{lem:Ziterate}, we have $\alpha^K(c)\in Q$. Thus, \[\alpha^K(a)=\alpha^K(c+b)=\alpha^K(c)+\alpha^K(b)\in Q+Q.\] Let $k_0$ be large enough that $\alpha^{k_0}(Q+Q)\subseteq Q$. Then we have $\alpha^{K+k_0}(a)=\alpha^{k_0}(\alpha^K(a))\in Q$ so the result holds with $M=K+k_0$.
\end{proof}

Recall that two ideals $\frak a,\frak b$ in $\Z_n$ are equivalent (written $\frak a\sim\frak b$) if there exists a constant $k$ such that $n^k\frak a\subseteq \frak b$ and $n^k\frak b\subseteq \frak a$ (see Definition \ref{def:equivideal}).

\begin{lem} \label{lem:equividealsequalsubsets}
Let $\frak a,\frak b$ be ideals of $\Z_n$ such that $\frak a\sim\frak b$. Then $\mathcal C(\frak a)=\mathcal C(\frak b)$.
\end{lem}

\begin{proof}
By definition, the ideal $\frak b$ determines only the fractional parts of the elements in $\mathcal C(\frak b)$, and if $b=\dots b_2b_1 b_0\in\frak b$, then there are elements of $\mathcal C(\frak b)$ with fractional part $0.b_k\cdots b_0$ for each $k\geq 1$ and arbitrary integral part. From this description, it is clear that for any $k$, the elements \[b=\dots b_2b_1b_0 \qquad \textrm{ and } \qquad n^kb=\dots b_2b_1b_0\underbrace{0\dots 0}_{k \textrm{ times}}\] define the same set of fractional parts of elements in $\mathcal C(\frak b)$. Since there exists $k$ such that $n^k\frak{b}\subset \frak{a}$, we see that $\mathcal{C}(\frak{b})\subseteq \mathcal{C}(\frak{a})$. By a symmetric argument, we also have $\mathcal{C}(\frak{a})\subseteq \mathcal{C}(\frak{b})$.
\end{proof}

\begin{lem} \label{lem:confiningsubgroup}
For any ideal $\frak a\subseteq \Z_n$, the confining subset $\mathcal{C}(\mathfrak{a})$ is a subring of $H$.
\end{lem}

\begin{proof}
It follows from the proof of Lemma \ref{lem:idealtoconfining} that $\mathcal{C}(\mathfrak{a})$ is closed under addition. Moreover, by definition it is closed under additive inverses, and so $\mc C(\frak a)$ is an additive subgroup of $H$. It also contains the multiplicative identity $1$ by definition. It remains to be shown that it is closed under multiplication.

For this purpose it will be helpful to write elements of $\Z\left[\frac{1}{n}\right]$ in a slightly different form than their base $n$ expansions. Given any element of $\Z\left[\frac{1}{n}\right]$ we may write it, for \emph{any sufficiently large $k$}, as $(-1)^\delta(an^{-k}+x)$ where $x\in \Z$, $\delta\in\{\pm 1\}$, and $a=a_0+a_1n+\ldots+a_{k-1}n^{k-1}$ with each $a_i\in \{0,\ldots,n-1\}$. That is, $x$ is the integer part and $an^{-k}$ is the fractional part.

In particular, if $u,v\in \mc C(\mathfrak a)$ then we may write \[u=(-1)^\delta (an^{-k}+x) \text{ and } v=(-1)^\epsilon (bn^{-k}+y),\] where $\delta,\epsilon\in \{0,1\}$, $x,y\in \Z$, 
\[a=a_0+a_1n+\ldots+a_{k-1}n^{k-1} \text{ and } b=b_0+b_1n+\ldots+b_{k-1}n^{k-1},\]  and the digits $a_i$ agree with the first $k$ digits of an element of $\mathfrak{a}$ and similarly for the $b_i$. This is to say that there are elements \[a+zn^k,b+wn^k \in \mathfrak a \text{ where } z,w\in \Z_n.\] We aim to show that $uv\in \mc C(\mathfrak a)$. We have 
\[
uv=(-1)^{\delta+\epsilon}(ab+ayn^k+bxn^k)n^{-2k}+xy.
\]
 Thus the fractional part of %digits to the right of the decimal point in 
 $uv$ agrees with the first $2k$ digits of the integer $ab+ayn^k+bxn^k$ (note that this integer may have arbitrarily many digits in base $n$). To show that $uv\in \mc C(\mathfrak a)$, it suffices to show that the first $2k$ digits of $ab+ayn^k+bxn^k$ agree with the first $2k$ digits of some element of $\mathfrak a$.

To show this last fact we consider the elements $a+zn^k,b+wn^k\in \mathfrak a$. Since $\mathfrak a$ is an ideal it contains the element \[(a+zn^k)(b+wn^k)-(a+zn^k)wn^k-(b+wn^k)zn^k+(a+zn^k)yn^k+(b+wn^k)xn^k.\] Expanding this expression and canceling we have that \[ ab+ayn^k+bxn^k-zwn^{2k}+zyn^{2k}+wxn^{2k}\in \mathfrak a .\] The first $2k$ digits of this element agree with the first $2k$ digits of $ab+ayn^k+bxn^k$. Thus $uv \in \mc C(\mathfrak a)$ and the proof is complete.
\end{proof}

We now give a  more concrete description of the subring $\mc C(\mathfrak a)$, which will be useful in the following subsection. By Lemma \ref{lem:equivalentfull}, there is a \emph{full} ideal $\mathfrak b$ of $\Z_n$ with $\mathfrak a \sim \mathfrak b$, and by Lemma \ref{lem:equividealsequalsubsets} we have $\mc C(\mathfrak a)= \mc C(\mathfrak b)$. Hence to describe $\mc C(\mathfrak a)$ explicitly we may assume that $\mathfrak a$ itself is full. For ease of notation we may suppose that \[\mathfrak a = \Z_{p_1^{n_1}} \times \cdots \times \Z_{p_r^{n_r}} \times 0 \times \cdots \times 0.\] Set $l=p_1^{n_1}p_2^{n_2} \cdots p_r^{n_r}$ if $r>0$ and $l=1$ if $r=0$.

\begin{prop}\label{prop:Z1/l}
We have $\mc C(\mathfrak a)=\Z\left[\frac{1}{l}\right]$ as a subring of $\Z\left[\frac{1}{n}\right]$.
\end{prop}

\begin{proof}
First we show that $\Z\left[\frac{1}{l}\right]\subset \mc C(\mathfrak a)$. Since $\mc C(\mathfrak a)$ is a subring which contains $\Z$, it suffices to show that $\frac{1}{l}\in \mc C(\mathfrak a)$. Set $q=p_{r+1}^{n_{r+1}}\cdots p_k^{n_k}$ so that $l=\frac{n}{q}$. Thus we need to show that $\frac{1}{l}=\frac{q}{n}\in \mc C(\mathfrak a)$. From this equality we see that $\frac{1}{l}$ is $0.q$ in base $n$. Hence it suffices to show that $\mathfrak a$ contains an element of $\Z_n$ whose ones digit is $q$ when written in base $n$.

%For this purpose consider the isomorphism of $\Z_n$ with $\Z_{p_1^{n_1}} \times \cdots \times \Z_{p_k^{n_k}}$. The ideal $\mathfrak{a}$ is identified with $\Z_{p_1^{n_1}} \times \cdots \times \Z_{p_r^{n_r}}\times 0 \times \cdots \times 0$ via this isomorphism. We consider the unique element $x$ of $\Z_n$ which is mapped to $(q,\ldots,q,0,\ldots,0)$ (with exactly $r$ non-zero entries) by the isomorphism.  Note that $x$ is necessarily an element of $\frak a$; we will show that the ones digit of $x$ is $q$. We reduce $x$ modulo $n$ to obtain the element $\overline{x}\in \Z/n\Z$. Via the isomorphism of $\Z/n\Z$ with $\Z/p_1^{n_1}\Z \times \cdots \times \Z/p_k^{n_k}\Z $ the residue class of $\overline{q}$ in $\Z/n\Z$ is the unique element which maps to $(\overline{q},\ldots,\overline{q},0,\ldots,0)$. Since $\overline{x}$ also maps to this element, we must have $\overline{x}=\overline{q}$ in $\Z/n\Z$. Thus, the ones digit of $x$ is $q$. Therefore $x$ is an element of $\mathfrak a$ with ones digit $q$, as desired. 

We have the following commutative diagram, where the horizontal maps are isomorphisms, the vertical map $\phi_1\colon\Z_n\to \Z/n\Z$ on the left  is the ``reduction mod $n$" map defined in \eqref{eqn:phil} which sends $a=\ldots a_2a_1a_0\in \Z_n$ to $[a]_n=a_0\in \Z/n\Z$, and the vertical map on the right is the product of the ``reduction mod $p_i^{n_i}$''  maps $\phi_1\colon \Z_{p_i^{n_i}}\to \Z/p_i^{n_i}\Z$.  Consider the unique element $x$ of $\Z_n$ whose image in $\Z_{p_1^{n_1}}\times \cdots \times \Z_{p_k^{n_k}}$ is $(q,\ldots, q, 0, \ldots, 0)$, with exactly $r$ non-zero entries.  Note that $x$ is necessarily an element of $\frak a$; we will show that the ones digit of $x$ is $q$.  Applying $\phi_1$ we obtain the element $[x]_n\in \Z/n\Z$, and applying the product of the maps $\phi_1$ to $(q,\ldots, q, 0, \ldots, 0)$ yields the element $([q]_{p_1^{n_1}},\ldots,[q]_{p_r^{n_r}},0,\ldots,0)$.  Since the diagram commutes, $[x]_n$ must be the unique element of $\Z/n\Z$ which maps to this element.  As it is clear that $[q]_n\in \Z/n\Z$ also maps to this element, we must have $[x]_n=[q]_n$.  Thus the ones digit of $x$ is $q$, as desired. 

\begin{center}
\begin{tikzcd}[contains/.style = {draw=none,"\in" description,sloped}
	,/tikz/column 2/row 1/.append style={anchor=base east}
	,/tikz/column 2/row 4/.append style={anchor=base east}
	]
x \ar[d,contains]\ar[ddd,mapsto,bend right, shift right=2]& (q,\ldots, q,0,\ldots, 0) \ar[l, mapsto] \ar[d, contains] \ar[ddd,mapsto,bend left, shift left=17] \\
\Z_n \arrow[r, "\simeq"] \arrow[d, "\phi_1"]
& \Z_{p_1^{n_1}}\times \cdots \times \Z_{p_k^{n_k}} \arrow[d, "\phi_1\times \cdots\times \phi_1"] \\
 \Z/n\Z \arrow[r, "\simeq"] &  \Z/p_1^{n_1}\Z \times \cdots \times \Z/p_k^{n_k}\Z \\
%  [x]_n \ar[u,contains] 
{[x  ]}_n\ar[u,contains] & ([q]_{p_1^{n_1}},\ldots,[q]_{p_r^{n_r}},0,\ldots,0)  \ar[u,contains] \\
\end{tikzcd}
\end{center}

%Now we show that $\mc C(\mathfrak a)\subset \Z\left[\frac{1}{l}\right]$. Consider an element $(-1)^\delta x_r\cdots x_0. a_s \cdots a_1 a_0 \in \mc C(\mathfrak a)$, where $\delta\in\{\pm 1\}$. By definition of $\mc C(\mathfrak a)$ there is an element $a=\ldots a_s \ldots a_1 a_0\in \frak a$. Recall that $\mathfrak a$ is identified with $\Z_{p_1^{k_1}}\times \cdots \times \Z_{p_r^{k_r}} \times 0\times \cdots \times 0$. We may consider $a$ modulo $n^s$. Under the isomorphism of $\Z/n^s\Z$ with $\Z/p_1^{n_1s}\Z \times \cdots \times \Z/p_k^{n_ks}\Z$, its residue class maps to 0 in $\Z/p_i^{n_is}\Z$ for each $i>r$. The residue class of $a$ modulo $n^s$ is equal to the residue class of $a_s\cdots a_0$ modulo $n^s$. Thus, $a_s\cdots a_0$ is divisible by $p_i^{k_is}$ for each $i>r$ and therefore it is also divisible by $p_{r+1}^{n_{r+1}s}\cdots p_k^{n_ks}=q^s$. Write $a_s\cdots a_0=q^sy$ for some $y\in\Z$. We therefore have that \[(-1)^\delta x_r\cdots x_0.a_s\cdots a_0=(-1)^\delta \left(x_r\cdots x_0 + \frac{q^sy}{n^s}\right)=(-1)^\delta \left(x_r\cdots x_0 + \frac{y}{l^s}\right).\] This lies in $\Z\left[\frac{1}{l}\right]$ and since the element of $\mc C(\mathfrak a)$ that we started with was arbitrary, this shows that $\mc C(\mathfrak a)\subset \Z\left[\frac{1}{l}\right]$ as claimed.

Now we show that $\mc C(\mathfrak a)\subset \Z\left[\frac{1}{l}\right]$. Consider an element $(-1)^\delta x_r\cdots x_0. a_s \cdots a_1 a_0 \in \mc C(\mathfrak a)$, where $\delta\in\{\pm 1\}$. By definition of $\mc C(\mathfrak a)$ there is an element $a=\ldots a_s \ldots a_1 a_0\in \frak a$. Recall that $\mathfrak a$ is identified with $\Z_{p_1^{n_1}}\times \cdots \times \Z_{p_r^{n_r}} \times 0\times \cdots \times 0$.  We consider the analogous commutative diagram as above, but with vertical maps given by (products of) $\phi_{s+1}$.  Since $a\in \frak a$, the $\Z/p_{i}^{n_i(s+1)}\Z$--component of the image of $\phi_{s+1}(a)$ in $ \Z/p_1^{n_1(s+1)}\Z \times \cdots \times \Z/p_k^{n_k(s+1)}\Z$ is $0$ for each $i>r$.  Note that $\phi_{s+1}(a)=\phi_{s+1}(a_s\ldots a_1a_0)$.    Thus, $a_s\cdots a_0$ is divisible by $p_i^{n_i(s+1)}$ for each $i>r$ and therefore it is also divisible by $p_{r+1}^{n_{r+1}(s+1)}\cdots p_k^{n_k(s+1)}=q^{s+1}$. Write $a_s\cdots a_0=q^{s+1}y$ for some $y\in\Z$. We therefore have that \[(-1)^\delta x_r\cdots x_0.a_s\cdots a_0=(-1)^\delta \left(x_r\cdots x_0 + \frac{q^{s+1}y}{n^{s+1}}\right)=(-1)^\delta \left(x_r\cdots x_0 + \frac{y}{l^{s+1}}\right),\] and so this element lies in $\Z\left[\frac{1}{l}\right]$.   Since the element of $\mc C(\mathfrak a)$  we started with was arbitrary, this shows that $\mc C(\mathfrak a)\subset \Z\left[\frac{1}{l}\right]$ as claimed.
\end{proof}

\subsubsection{Actions on Bass-Serre trees} \label{section:bass-serre}

 In this subsection, we give an explicit geometric description of the action of $BS(1,n)$ on the Cayley graph $\Gamma(BS(1,n),\mc C(\frak a)\cup \{t^{\pm1}\})$ for any ideal $\frak a\subseteq \Z_n$.
 We begin by considering a particular ascending HNN extension of $\mc C(\mathfrak{a})$: \[G(\mathfrak{a})=\langle \mc C(\mathfrak{a}), s \mid sxs^{-1} = \alpha(x) \text{ for } x\in \mc C(\alpha)\rangle.\] 

\begin{lem}
For any ideal $\frak a\subseteq \Z_n$,  \[BS(1,n)\cong G(\mathfrak{a}).\]
\end{lem}

\begin{proof}
By Remark \ref{rem:Zinsubset}, $\Z\subseteq \mc C(\frak a)$. There is an obvious homomorphism $f:G(\mathfrak{a})\to BS(1,n)$ defined by \[x\mapsto x \text{ for } x\in \mc C(\mathfrak{a}), \qquad s\mapsto t.\] This homomorphism is surjective because $BS(1,n)$ is generated by $\Z\subseteq \mc C(\mathfrak{a})$ and $t$. We now show that $f$ is injective. Let $g \in \ker(f)$. By Lemma \ref{lem:movegenerators} we can find an expression of $g$ as a minimal length word in the generating set $\mc C(\frak a)\cup\{s^{\pm1}\}$ of the form  \[g=s^{-i}(x_1+\cdots+x_w)s^j,\] where $i,j\geq 0$ and $x_i\in \mc C(\frak a)$.  By Lemma \ref{lem:confiningsubgroup} we may write $x_1+\dots+x_w=x\in \mathcal C(\frak a)$, and the result is that $g=s^{-i}xs^j$.
As $g$ is in the kernel of $f$, we have \[1=f(g)=t^{-i}xt^j=\alpha^{-i}(x)t^{j-i}.\] Since $\alpha^{-i}(x)\in H$, we obtain a contradiction unless $j=i$. In this case we have \[f(g)=\alpha^{-i}(x)=0\] in $H$, and since $\alpha$ is an automorphism, $x=0$. But then \[g=s^{-i} 0 s^i=\alpha^{-i}(0)=0\] in $G(\mathfrak{a})$. This proves the statement.
\end{proof}

Hence we have a description of $BS(1,n)$ as an HNN extension over the additive subgroup $\mc C(\mathfrak{a})\leq H$, and therefore an action of $BS(1,n)$ on the standard Bass-Serre tree associated to  this HNN extension. Denote this tree by $T(\mathfrak{a})$. For the statement of the next two results recall that \emph{equivalence} of hyperbolic actions means equivalence up to coarsely equivariant quasi-isometry.

\begin{prop} \label{prop:bstreeequiv}
Let $G$ be a group which may be expressed as an ascending HNN extension 
\[
A*_A = \langle A, s\mid sas^{-1} =\phi(a) \text{ for all } a \text{ in } A\rangle,
\]
 where $A$ is a group and $\phi$ is an endomorphism of $A$. Then the action of $G$ on the Bass-Serre tree associated to this HNN extension is equivalent to its action on $\Gamma(G,A\cup \{s^{\pm 1}\})$.
\end{prop}

Before turning to the proof,  we note one immediate corollary.

\begin{cor}
The action of $BS(1,n)$ on $\Gamma(BS(1,n),\mc C(\mathfrak{a})\cup \{t^{\pm 1}\})$ is equivalent to the action of $BS(1,n)$ on $T(\mathfrak{a})$.
\end{cor}

\begin{proof}[Proof of Proposition \ref{prop:bstreeequiv}]
We apply the standard Schwarz--Milnor Lemma (see, e.g., \cite[Lemma~3.11]{ABO}).

Denote by $T$ the Bass-Serre tree associated to this HNN extension, which  may be described as follows. The vertices of $T$ are the left cosets of $A$ in $G$, and two cosets $gA$ and $hA$ are joined by an edge if \[gA=hxsA \quad\text{ or }\quad gA=hxs^{-1}A \quad \text{ for some } x\in A.\]

Consider the vertex $v=A$ and the edge $E=[A,sA]$ containing $v$. Clearly we have $\bigcup_{g\in G} gE=T$. Hence by the Schwarz--Milnor Lemma, the action of $G$ on $T$ is equivalent to the action of $G$ on $\Gamma(G,S)$ where $S=\{g\in G: d(v,gv)\leq 3\}$. Note that $A\subset S$ since it fixes the vertex $v$, and thus $[S\cup \{s^{\pm 1}\}]\preccurlyeq [A\cup \{s^{\pm 1}\}]$. We will show that also $[A\cup \{s^{\pm 1}\}]\preccurlyeq [S\cup \{s^{\pm 1}\}]$, which will prove the proposition.

By the description of the vertices of $T$ as cosets of $A$, any vertex in the radius 3 neighborhood of $v$ has  one of the following forms:
\begin{itemize}
\item $xs^\delta v$ where $x\in A$ and $\delta \in \{\pm 1\}$; 
\item $x_1s^{\delta_1} x_2 s^{\delta_2} v$ where $x_i\in A$ %for $i=1,2$ 
and $\delta_i\in \{\pm 1\}$ for $i=1,2$; or 
\item $x_1s^{\delta_1}x_2 s^{\delta_2}x_3s^{\delta_3}v$ where $x_i \in A$ %for $i=1,2,3$ 
and $\delta_i \in \{\pm 1\}$ for $i=1,2,3$.
\end{itemize}

If $g\in S$,  it therefore sends $v$ to a vertex of one of the above three forms. We deal with the last case explicitly, showing that $g$ has bounded word length in the generating set $A\cup \{s^{\pm 1}\}$. The other two cases are entirely analogous.

If $gv=x_1s^{\delta_1}x_2s^{\delta_2}x_3s^{\delta_3}v$ then \[(x_1 s^{\delta_1}x_2 s^{\delta_2} x_3 s^{\delta_3})^{-1}g \in \operatorname{Stab}_G(v)=A.\] Hence \[g=x_1s^{\delta_1}x_2s^{\delta_2}x_3s^{\delta_3}y\] for some $y\in A$, and this shows that $g$ has word length at most 7 in the generating set $A\cup \{s^{\pm 1}\}$.
\end{proof}

%We now give a  more concrete geometric description of the Bass-Serre tree $T(\mathfrak{a})$. 
 As in the previous subsection, we may assume that $\mathfrak a$  is a full ideal, and for ease of notation we may suppose that \[\mathfrak a = \Z_{p_1^{n_1}} \times \cdots \times \Z_{p_r^{n_r}} \times 0 \times \cdots \times 0.\] We again set $l=p_1^{n_1}p_2^{n_2} \cdots p_r^{n_r}$ and $q=\frac{n}{l}=p_{r+1}^{n_{r+1}}\cdots p_k^{n_k}$.  By Proposition \ref{prop:Z1/l}, we have that $\mc C(\mathfrak a)=\Z\left[\frac{1}{l}\right]$ as a subring of $\Z\left[\frac{1}{n}\right]$.  Thus the Bass-Serre tree $T(\mathfrak a)$ is the Bass-Serre tree of the ascending HNN extension of $\Z\left[\frac{1}{l}\right]$, where one map from the edge group to the vertex group is the identity and the other is multiplication by $n$. %glued to itself via multiplication by $n$. 
 Our final goal is to give a concrete geometric description of this tree and the associated action of $BS(1,n)$.

%\subsubsection{A geometric description of the Bass-Serre trees}

%We consider the ideal $\mathfrak{a}$ of $\Z_n$ and its associated subring $\mc C(\mathfrak a)=\Z\left[\frac{1}{l}\right]$ as described in the last subsection. 

%Recall that $l$ is a divisor $p_1^{n_1}\cdots p_r^{n_r}$ of $n=p_1^{n_1} \cdots p_k^{n_k}$. 
%As before we denote . 
We first describe an explicit action $BS(1,n)$ on a tree  $T'(\mathfrak{a})$ below.  We will then show that $T(\mathfrak{a})$ and $T'(\mathfrak{a})$ are  $BS(1,n)$--equivariantly isomorphic.

\begin{defn}\label{def:treeT'}
Let $T'(\frak a)$ be the tree with the following vertices and edges.
\begin{itemize}
\item The vertices of $T'(\mathfrak a)$ are identified with $\Q_q\times \Z$ up to an equivalence relation $\sim$;
\item for pairs $(x,h)$ and $(x',h')$ in $\Q_q\times \Z$ we have $(x,h)\sim (x',h')$ if and only if $h=h'$ and $||x-x'||_q\leq q^{-h}$, where $||\cdot ||_q$ denotes the $q$--adic absolute value on $\Q_q$  (see Definition \ref{def:absval}); and 
\item a vertex represented by $(x,h)\in \Q_q\times \Z$ is joined by an edge to the vertex represented by $(x,h+1)$.
\end{itemize}
We define an action of $BS(1,n)$ on $T'(\frak a)$ as follows:
\begin{itemize}
\item if $a$ denotes the normal generator 1 of $\Z\left[\frac{1}{n}\right]\leq BS(1,n)$, then the action of $a$ on the vertices of $T'(\mathfrak a)$ is given by \[a\colon (x,h)\mapsto (x+1,h);\] and
\item the generator $t$ acts on the vertices of $T'(\mathfrak a)$ by \[t\colon(x,h)\mapsto (nx,h+1).\]
\end{itemize}
\end{defn}

Here $h$ is meant to indicate a ``height'' and the equivalence relation reflects the fact that the tree distinguishes between more $q$-adic numbers, the larger the parameter $h$ is. The reader may check that the graph described above is indeed a tree and that the actions of $a$ and $t$ do indeed define an action of $BS(1,n)$. In fact, $T'(\mathfrak a)$ is just the regular $(q+1)$-valent tree. Before explaining why $T'(\mathfrak a)$ is equivariantly isomorphic to $T(\mathfrak a)$, we present two examples that the reader may find illustrative.
\\

\begin{ex}
First consider the group $BS(1,2)$. The ring $\Z_2$ has two full ideals: $(0)$ and $\Z_2$.  If $\frak a=\Z_2$, then $l=2$ and $q=1$, so the vertices of $T'(\Z_2)$ are identified with $\{0\}\times \Z$.  Hence the action $BS(1,2)\curvearrowright T'(\Z_2)$ is simply the standard action of $BS(1,2)$ on the line by translations.   If $\frak a=(0)$, then $l=1$ and $q=2$, so the vertices of $T(0)$ are identified with $\Q_2\times \Z$ up to the equivalence relation $\sim$.  This results in the main Bass-Serre tree of $BS(1,2)$ with the standard action, as shown on the left of Figure \ref{fig:bstrees}.   % The action $BS(1,2)\curvearrowright T'(0)$ is the main Bass-Serre tree of $BS(1,2)$ and is shown on the left of Figure \ref{fig:bstrees}. 
 In the figure, vertices are labeled by 2--adic numbers. Heights are implicit in the figure, with vertices at the same height in the figure having the same height in $\Z$. The height $h=0$ is illustrated with a dotted line. The generator $t$ acts loxodromically, shifting each vertex directly upward in the figure. For example, consider the vertex $(1.1,1)\in T'(0)$.  We have $t(1.1,1)=(2\cdot 1.1, 1+1)=(11,2)$.  The generator $a$ acts elliptically, where the action is via a 2--adic odometer $x\mapsto x+1$ on $\Q_2$.
 \end{ex}
 
\begin{center}
\begin{figure}[h]

\begin{tabular}{l l}

\begin{overpic}[width=0.5\textwidth,percent]{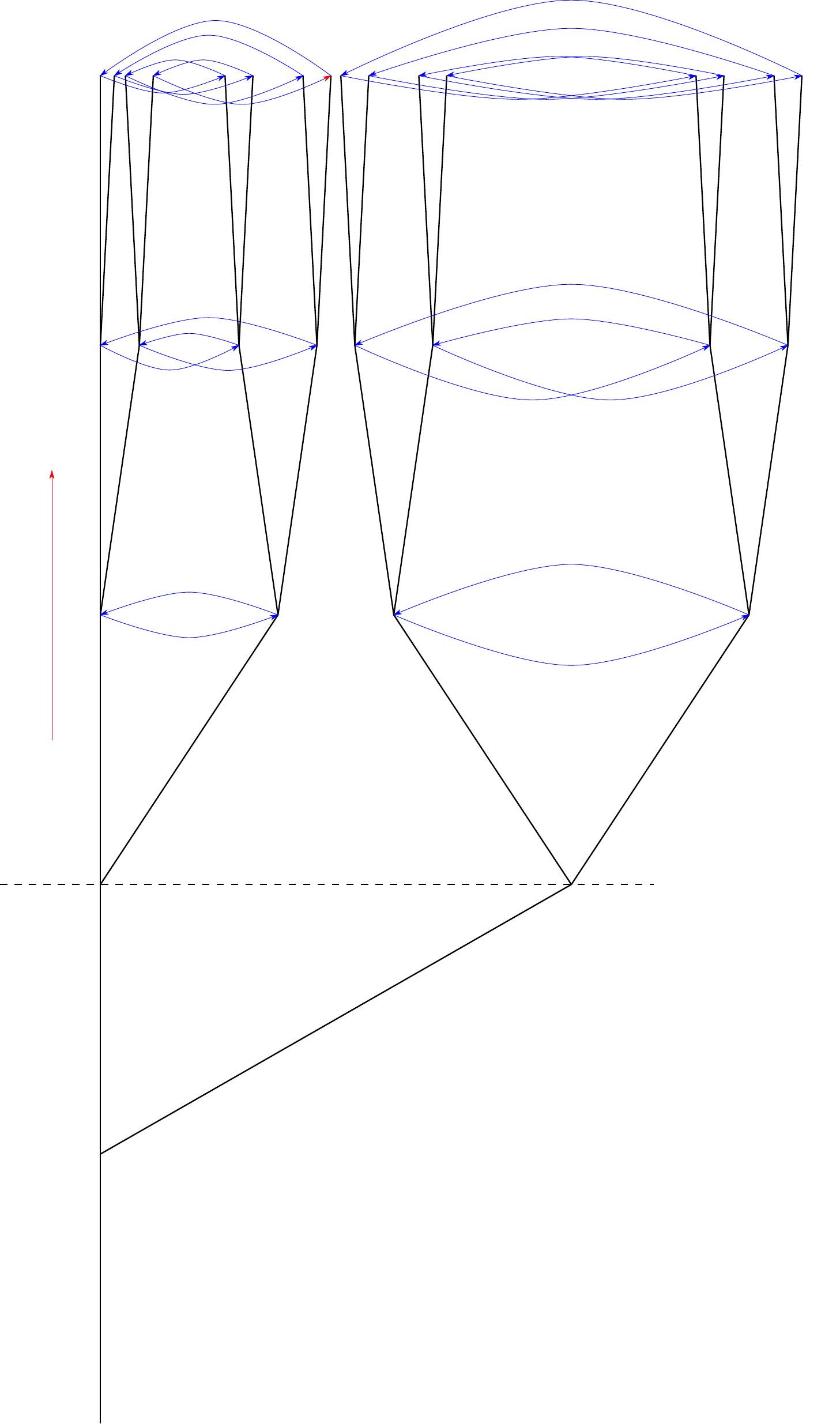}
\put (1,60) {\color{red} $t$}
\put (40,62) {\color{blue} $a$}

\put (10,19) {\scriptsize $0$}
\put (10,38) {\scriptsize $0$}
\put (42,38) {\scriptsize $0.1$}
\put (10,57) {\scriptsize $0$}
\put (21,57) {\scriptsize $1$}
\put (30,57) {\scriptsize $0.1$}
\put (54,57) {\scriptsize $1.1$}
\put (8,75) {\scriptsize $0$}
\put (11,75) {\scriptsize $10$}
\put (18,75) {\scriptsize $1$}
\put (22.5,75) {\scriptsize $11$}
\put (26,75) {\scriptsize $0.1$}
\put (31,75) {\scriptsize $10.1$}
\put (51,75) {\scriptsize $1.1$}
\put (56,75) {\scriptsize $11.1$}
\put (7,97) {\tiny $0$}
\put (8,95.5) {\tiny $100$}
\put (9,97) {\tiny $10$}
\put (11,95.5) {\tiny $110$}
\put (15,97) {\tiny $1$}
\put (17,95.5) {\tiny $101$}
\put (21,97) {\tiny $11$}
\put (22.5,95.5) {\tiny $111$}
\put (24,97) {\tiny $0.1$}
\put (26,95.5) {\tiny $100.1$}
\put (28.5,97) {\tiny $10.1$}
\put (31.5,95.5) {\tiny $110.1$}
\put (48,97) {\tiny $1.1$}
\put (50.5,95.5) {\tiny $101.1$}
\put (54,97) {\tiny $11.1$}
\put (56,95.5) {\tiny $111.1$}
\end{overpic}

&

\begin{overpic}[width=0.5\textwidth,percent]{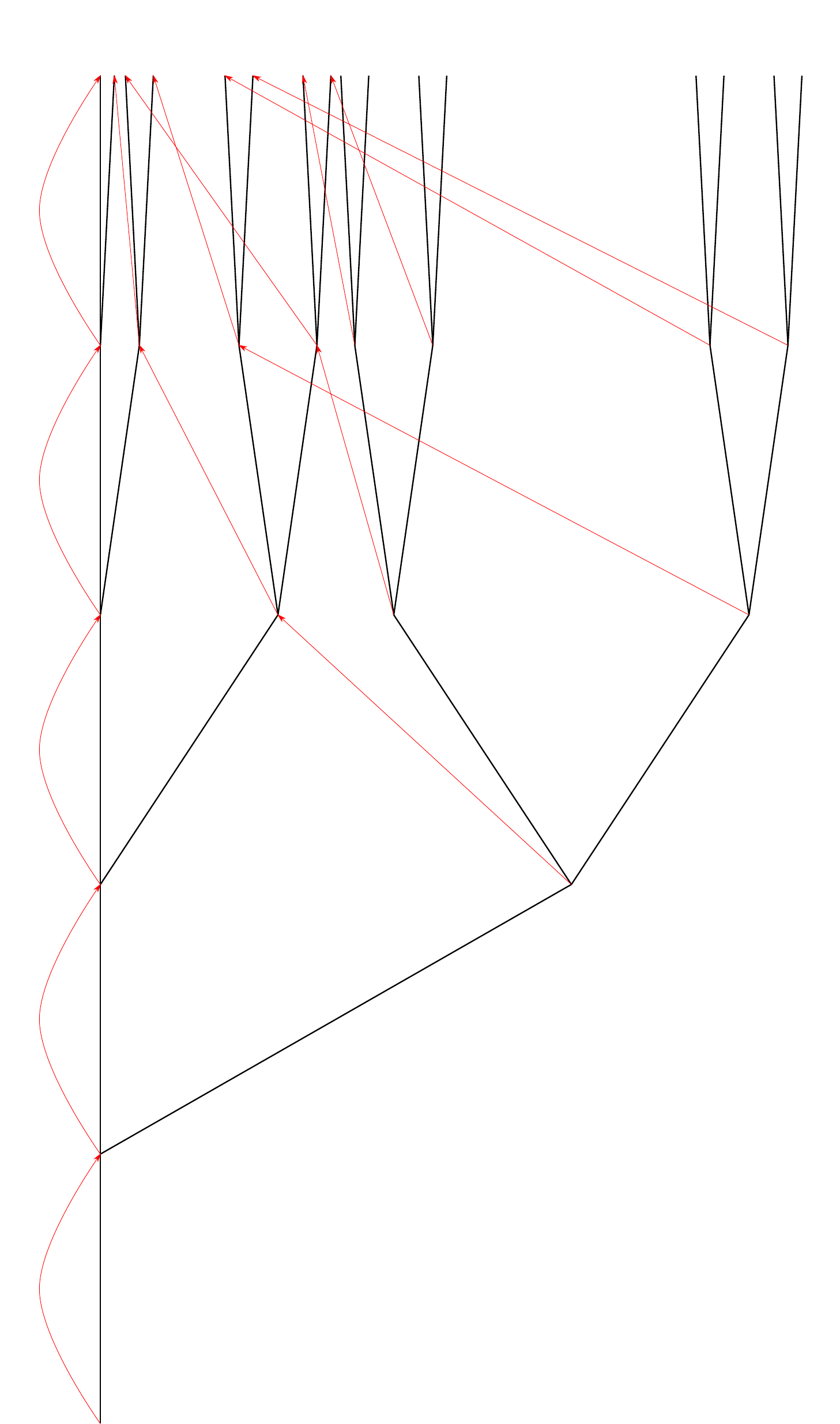}
\put (1,60) {\color{red} $t$}

\put (10,19) {\scriptsize $0$}
\put (10,38) {\scriptsize $0$}
\put (42,38) {\scriptsize $0.1$}
\put (10,57) {\scriptsize $0$}
\put (21,57) {\scriptsize $1$}
\put (30,57) {\scriptsize $0.1$}
\put (54,57) {\scriptsize $1.1$}
\put (8,75) {\scriptsize $0$}
\put (11,75) {\scriptsize $10$}
\put (18,75) {\scriptsize $1$}
\put (22.5,75) {\scriptsize $11$}
\put (26,75) {\scriptsize $0.1$}
\put (31,75) {\scriptsize $10.1$}
\put (51,75) {\scriptsize $1.1$}
\put (56,75) {\scriptsize $11.1$}
\put (7,97) {\tiny $0$}
\put (8,95.5) {\tiny $100$}
\put (9,97) {\tiny $10$}
\put (11,95.5) {\tiny $110$}
\put (15,97) {\tiny $1$}
\put (17,95.5) {\tiny $101$}
\put (21,97) {\tiny $11$}
\put (22.5,95.5) {\tiny $111$}
\put (24,97) {\tiny $0.1$}
\put (26,95.5) {\tiny $100.1$}
\put (28.5,97) {\tiny $10.1$}
\put (31.5,95.5) {\tiny $110.1$}
\put (48,97) {\tiny $1.1$}
\put (50.5,95.5) {\tiny $101.1$}
\put (54,97) {\tiny $11.1$}
\put (56,95.5) {\tiny $111.1$}
\end{overpic}

\end{tabular}
\caption{The action $BS(1,2)\curvearrowright T'(0)$ (left) and the action $BS(1,6)\curvearrowright T'(\Z_2\times 0)$ (right).}
\label{fig:bstrees}
\end{figure}

\end{center}

\begin{ex}
Now consider the group $BS(1,6)$.  The ring $\Z_6$ has four full ideals: $(0)$, $\Z_2\times (0)$, $(0)\times \Z_3$, and $\Z_2\times \Z_3$. The trees $T'(0)$ and $T'(\Z_2\times \Z_3)$ are the main Bass-Serre tree and the standard action on the line, respectively, as in the previous example. We will describe the action $BS(1,6)\curvearrowright T'((0)\times \Z_3)$. The action $BS(1,6)\curvearrowright T'(\Z_2\times (0))$ may be described in a similar way.

If $\frak a=(0)\times \Z_3$, then $l=3$ and $q=2$, hence the vertex set of the tree $T'((0)\times \Z_3)$ is identified with $\Q_2\times\Z$ up to the equivalence relation $\sim$.  In particular, the tree $T'((0)\times \Z_3)$ for $BS(1,6)$ is isomorphic to the main Bass-Serre tree for $BS(1,2)$.  However, the action is different. The generator $a$ is still elliptic, acting as a 2--adic odometer on $\Q_2$, just as it did for $BS(1,2)$. The generator $t$ is also still loxodromic.   However, $t$ has a more complicated action than simply ``shifting vertices directly upward.'' This is because $t$ now acts as multiplication by 6 on $\Q_2$ while simultaneously increasing heights by 1, while in the previous example it acted as multiplication by $2$ on $\Q_2$ while increasing heights by $1$.  For example, consider the same vertex $(1.1,1)$ as in the previous example, but now as a vertex of $T'((0)\times \Z_3)$.  Recall that in the previous example, applying $t\in BS(1,2)$ sent $(1.1,1)$ to $(11,2)$.  Here, applying $t\in BS(1,6)$ yields 
\[
t(1.1,1)=(6\cdot 1.1,1+1)=(1001,2)\sim (1,2).
\]
The action of $t$ on certain vertices of the tree is illustrated via red arrows in Figure~\ref{fig:bstrees}.
\end{ex}

\begin{prop}
The trees $T(\mathfrak a)$ and $T'(\mathfrak a)$ are $BS(1,n)$--equivariantly isomorphic. 
\end{prop}

\begin{proof}
This is a standard exercise in Bass-Serre theory. The group $BS(1,n)$ acts on $T'(\mathfrak a)$ with a single orbit of vertices and a single orbit of edges. Note that the stabilizer of any vertex is contained in $\Z\left[\frac{1}{n}\right]\trianglelefteq BS(1,n)$ since $t$ acts loxodomically.  Moreover, $\Z\left[\frac{1}{n}\right]$ permutes the vertices at any given height in the tree. For $y\in \Z\left[\frac{1}{n}\right]$ the action on vertices at height $h$ is given explicitly by $(x,h)\mapsto (x+y,h)$. Hence $y\in \Z\left[\frac{1}{n}\right]$ fixes (the equivalence class of) the vertex $(x,h)$ if and only if $||y||_q\leq q^{-h}$. The additive subgroup of $\Z\left[\frac{1}{n}\right]$ consisting of elements of $q$--adic absolute value at most $q^{-h}$ is given exactly by $q^h \Z\left[\frac{1}{l}\right]=n^h\Z\left[\frac{1}{l}\right]$. Choosing a vertex at height 0 and an adjacent vertex at height 1, the vertex stabilizers are $\Z\left[\frac{1}{l}\right]$ and $n\Z\left[\frac{1}{l}\right]$, respectively. Thus, the quotient graph of groups for $BS(1,n)$ has a single vertex group $\Z\left[\frac{1}{l}\right]$. The edge group embeds isomorphically into $\Z\left[\frac{1}{l}\right]$ on one end and embeds as the subgroup $n\Z\left[\frac{1}{l}\right]$ on the other end. This is the same as the graph of groups defining $T(\mathfrak a)$, so the two trees are equivariantly isomorphic.
\end{proof}

\subsection{Subsets confining under the action of $\alpha^{-1}$}

Let $Q^-$ be the subset defined in (\ref{eqn:Q-}).

\begin{lem} \label{lem:alpha^-1largest}
Let $Q\subseteq H$ be confining under the action of $\alpha^{-1}$. Then $\alpha^{-k}(Q^-)\subseteq Q$ for some $k\geq 0$.
\end{lem}

\begin{proof}
Let $k_0\in \Z_{\geq 0}$ be large enough that $\alpha^{-k_0}(Q+Q)\subseteq Q$. We may suppose that $k_0>0$, for otherwise $Q$ is a subgroup of $H$ and it is easy to check that in fact $Q=H$. Since $H=\bigcup_{k\in \Z_{\geq 0}} \alpha^k(Q)$, we may also choose $k_1$ large enough that $\alpha^{-k_1}(a)=an^{-k_1}\in Q$ for any $a\in \{0,\ldots, n-1\}$.

We claim that any number of the form \[\sum_{i=0}^r a_i n^{-k_1-(i+1)k_0} \text{ with } a_i \in \{0,\ldots,n-1\} \text{ for all } 0\leq i\leq r\] (for any $r\geq 0$) lies in $Q$. In other words, any number \[0.\underbrace{0\cdots \cdots \cdots0}_{k_1+k_0-1 \text{ times}} a_0 \underbrace{0\cdots \cdots 0}_{k_0 -1\text{ times}} a_1 \underbrace{0\cdots \cdots 0}_{k_0-1 \text{ times}}a_2 \underbrace{0\cdots \cdots 0}_{k_0-1 \text{ times}} \cdots a_r\] between 0 and 1 which may be written in base $n$ with
\begin{itemize}
\item $k_1+k_0-1$ 0's after the decimal point and then
\item $k_0-1$ 0's between any consecutive potentially nonzero digits
\end{itemize}
\noindent lies in $Q$. We will prove this by induction on $r$. The base case, when $r=0$, follows since $a_0 n^{-k_1}\in Q$ for any $a_0 \in \{0,\ldots,n-1\}$ and therefore $\alpha^{-k_0}(a_0 n^{-k_1})=a_0 n^{-k_1-k_0}\in Q$ because $Q$ is closed under the action of $\alpha^{-1}$. Suppose that the claim is true for all $r<s$. Consider a number $x=\sum_{i=0}^s a_i n^{-k_1-(i+1)k_0}$ with $a_i\in \{0,\ldots,n-1\}$ for all $i$. We may write \[x=a_0n^{-k_1-k_0}+\sum_{i=1}^s a_in^{-k_1-(i+1)k_0}.\] We have that $a_0n^{-k_1}\in Q$ and by induction, \[\sum_{i=1}^s a_i n^{-k_1-ik_0}=\sum_{i=0}^{s-1} a_{i+1}n^{-k_1-(i+1)k_0} \in Q.\] Hence, \[x= a_0n^{-k_1-k_0}+\sum_{i=1}^s a_in^{-k_1-(i+1)k_0}=\alpha^{-k_0}\left(a_0n^{-k_1}+\sum_{i=1}^s a_in^{-k_1-ik_0}\right)\in Q.\] This proves the claim.

Now consider a number of the form $x=\sum_{i=0}^r a_i n^{-k_1-k_0-i}$. In other words, \[x=0.\underbrace{0\cdots\cdots\cdots 0}_{k_1 +k_0-1 \text{ times}}a_0 a_1 \cdots a_r.\] We may write \[x=\sum_{j=0}^{k_0-1} \sum_{\substack{ i\in \Z_{\geq 0} \\ j+ik_0\leq r}} a_{j+ik_0} n^{-(k_1+(i+1)k_0+j)}=\sum_{j=0}^{k_0-1} x_j,\] where \[x_j=\sum_{\substack{ i\in \Z_{\geq 0} \\ j+ik_0\leq r}} a_{j+ik_0} n^{-(k_1+(i+1)k_0+j)}.\] In other words, we are writing \[\begin{tabular}{l l l l l l l l l l l l l l l l l}
$x$ & $=$ & 0& $.$ & 0& $\cdots$ & 0 & $a_0$ &0 & 0 & $\cdots$ & 0 & 0 & $a_{k_0}$ & 0 &0 & $\cdots$ \\
& $+$ & 0& $.$ & 0& $\cdots$&  0 & 0 & $a_{1}$  &0 & $\cdots$ & 0 &  0& 0 & $a_{1+k_0}$&0 & $\cdots$ \\
& $+$ & 0 & $.$ & 0 & $\cdots$ & 0 & 0 & 0 & $a_{2}$ & $\cdots$ & 0 & 0 & 0 & 0 & $a_{2+k_0}$ & $\cdots$ \\ 
& $+$ & $\cdots$ & &  & &  & & & & & & & & & & \\
& $+$ & 0 & $.$ & 0 & $\cdots$ & 0 & 0 & 0 & 0 & $\cdots$ & 0 & $a_{(k_0-1)}$ &0 & 0 & 0 & $\cdots$ \\
&\end{tabular}\] with the summands being $x_0,x_1,x_2,\ldots,$ and $x_{k_0-1}$, respectively. Notice that \[x_j=\sum_{\substack{i \in \Z_{\geq 0} \\j+ik_0\leq r}} a_{j+ik_0} n^{-(k_1+(i+1)k_0+j)}=\alpha^{-j}\left(\sum_{\substack{i \in \Z_{\geq 0} \\j+ik_0\leq r}} a_{j+ik_0} n^{-(k_1+(i+1)k_0)}\right)=\alpha^{-j}(y_j),\] where \[y_j=\sum_{\substack{i \in \Z_{\geq 0} \\j+ik_0\leq r}} a_{j+ik_0} n^{-(k_1+(i+1)k_0)}\] and, by our claim, $y_j\in Q$ for each $j\in \{0,\ldots,k_0-1\}$. Since $Q$ is closed under the action of $\alpha^{-1}$, we have $x_j\in Q$ for each $j$. Therefore $x\in Q^{k_0}$. Using that $\alpha^{-k_0}(Q+Q)\subseteq Q$, we have $\alpha^{-k_0^2}(x)=\alpha^{-k_0 \cdot k_0}(x)\in Q$.

Finally, for any \textit{positive} number $y=\sum_{i=0}^r a_i n^{-i-1} \in Q^-$, we have \[\alpha^{-k_1-k_0+1}(y)=\sum_{i=0}^r a_i n^{-k_1-k_0-i}.\] By our work in the last paragraph, this proves that \[\alpha^{-k_0^2}(\alpha^{-k_1-k_0+1}(y))\in Q.\] In other words, $\alpha^{-k_0^2-k_0-k_1+1}(y)\in Q$.

We have shown that $\alpha^{-k}(y)\in Q$ for any \textit{positive} $y\in Q^-$ where $k=k_0^2+k_0+k_1-1$. Since $Q^-$ and $Q$ are symmetric, this proves that for any \textit{negative} $y\in Q^-$ we have $\alpha^{-k}(y)=-\alpha^{-k}(-y)\in Q$. In other words, $\alpha^{-k}(Q^-)\subset Q$ where $k=k_0^2+k_0+k_1-1$.
\end{proof}

\begin{prop} \label{prop:singlealpha^-1}
Let $Q\subseteq H$ be strictly confining under the action of $\alpha^{-1}$. Then $[Q\cup\{t^{\pm 1}\}]=[Q^-\cup\{t^{\pm 1}\}]$.
\end{prop}

\begin{proof}
Lemma \ref{lem:alpha^-1largest}, we have $[Q\cup \{t^{\pm 1}\}] \preccurlyeq [Q^-\cup \{t^{\pm 1}\}]$. Suppose that the inequality is strict. We will show then that $[Q\cup \{t^{\pm 1}\}]=\left[H\cup \{t^{\pm 1}\}\right]$ and this will contradict that $Q$ is \textit{strictly} confining.

By Lemma \ref{lem:alpha^-1largest} and Lemma \ref{lem:elements}, we may suppose that $Q^-\subseteq Q$. If there exists $k$ large enough that $\alpha^{-k}(x)\in Q^-$ for all $x\in Q$, then we have $[Q\cup\{t^{\pm 1}\}]=[Q^-\cup\{t^{\pm 1}\}]$, as desired.

Otherwise, there exist numbers $x=a_r\cdots a_0.a_{-1} \cdots a_{p(x)}\in Q$ with $r$ arbitrarily large and $a_r\neq 0$. We will prove in this case that $n^t\in Q$ for $t$ arbitrarily large. This will complete the proof, for in this case, for any $a\in \{0,\ldots,n-1\}$ we will also have $an^t\in Q$ for any $t$, by the standard arguments. Hence given a positive number $y=\sum_{i=-p}^q a_i n^i\in H$ we have $\sum_{i=-p}^q a_i n^{i+k_0(p+q+1)}\in Q^{p+q+1}$ and therefore \[\alpha^{-(p+q+1)k_0}\left(\sum_{i=-p}^q a_i n^{i+k_0(p+q+1)}\right)=y\in Q.\] Since $Q$ is symmetric, this will prove that $H\subseteq Q$.

Consider a number $x=a_r\cdots a_0.a_{-1} \cdots a_{p(x)}\in Q$ with $r\gg 0$ (how large will be determined later in the proof). We clearly have $n^r \leq x<n^{r+1}$. For any $s>0$ we have \[n^sx=a_r\cdots a_0 \cdots a_{-s}. a_{-s-1}\cdots a_{p(x)}\geq n^{r+s}=1\underbrace{0\ldots\dots 0}_{r+s \text{ times}}\] (with the decimal point to the right of $a_{p(x)}$ if $s>p(x)$). We claim that there exists $c\in \{0,1,\ldots, n^s\}$ such that \[cx=1\underbrace{0\cdots\cdots 0}_{s-2 \text{ times}} b_r b_{r-1} \cdots b_0.b_{-1} \cdots b_{p(cx)}.\]  If this is not the case, then since $n^sx \geq n^{r+s}$ there exists some $c\in\{0,\ldots,n^s\}$ with \[cx<1\underbrace{0\ldots\dots\dots 0}_{r+s-1 \text{ times}}\text{ but } (c+1)x\geq 1\underbrace{0\ldots\dots 0}_{s-3 \text{ times}} 1 \underbrace{0 \ldots\dots 0}_{r+1 \text{ times}}.\] From these inequalities we see that \[x=(c+1)x-cx \geq 1\underbrace{0 \ldots\dots 0}_{r+1 \text{ times}}=n^{r+1}\] which is a contradiction.

So choose $c\in \{0,1,\ldots,n^s\}$ with \[cx=1\underbrace{0\cdots\cdots 0}_{s-2 \text{ times}} b_r b_{r-1} \cdots b_0.b_{-1} \cdots b_{p(cx)}.\] Assuming that $r> n^sk_0$, we have \[\alpha^{-ck_0}(cx)=1\underbrace{0\cdots\cdots 0}_{s-2 \text{ times}} b_r b_{r-1} \cdots b_{ck_0}. b_{ck_0-1}\cdots b_0b_{-1} \cdots b_{p(cx)}\in Q.\] Since $Q$ is closed under the action of $\alpha^{-1}$, we then also have \[1\underbrace{0\cdots\cdots 0}_{s-2 \text{ times}}.b_r b_{r-1} \cdots b_{p(cx)}\in Q.\] Since $Q^-\subseteq Q$, we then have \[10\ldots 0=n^{s-2}\in Q+Q\] and therefore $n^{s-2-k_0}\in Q$. But since $s$ is arbitrarily large, we then have $n^t\in Q$ for $t$ arbitrarily large. This completes the proof.

\end{proof}

\subsubsection{The action on $\hyp^2$}

There is a well-known action of $BS(1,n)$ on $\hyp^2$ defined via the representation \[BS(1,n)\to \operatorname{PSL}(2,\R),\] where \[a \mapsto \begin{pmatrix} 1 & 1 \\ 0 & 1 \end{pmatrix},\quad t\mapsto \begin{pmatrix} n^{1/2} & 0 \\ 0 & n^{-1/2} \end{pmatrix}.\] The restriction of this representation to $H=\langle t^k a t^{-k} \mid k\in \Z\rangle$ is given by \[r\mapsto \begin{pmatrix} 1 & r \\ 0 & 1 \end{pmatrix}\] where we identify $H$ as a subset of $\R$ in the usual way.

\begin{prop} \label{prop:planeaction}
The action of $BS(1,n)$ on $\hyp^2$ is equivalent to the action of $BS(1,n)$ on \\ $\Gamma(BS(1,n),Q^- \cup \{t^{\pm 1}\})$.
\end{prop}

\begin{proof}
We will apply the Schwarz-Milnor Lemma;  see, e.g.,  \cite[Lemma~3.11]{ABO}.

We consider the upper half-plane model of $\hyp^2$. We show first that the orbit of $i$ under $BS(1,n)$ is $(\log(n)+1)$--dense in $\hyp^2$. For this, note first that the orbit of $i$ under $\langle a \rangle$ is 1--dense in the horocycle $\{z\in \hyp^2\mid \Im(z)=1\}$. This follows easily from the fact that $d(i,ai)=d(i,1+i)<1$. Now, for any $k\in \Z$, $t^k$ takes the horocycle $\{z\in \hyp^2\mid  \Im(z)=1\}$ isometrically to the horocycle $\{z\in \hyp^2 \mid \Im(z)=n^k\}$. Hence the orbit of $i$ is 1--dense in the horocycle $\{z\in \hyp^2 \mid \Im(z)=n^k\}$ for any $k\in \Z$. Moreover, for any $k\in \Z$, the distance between the horocycles $\{z\in \hyp^2 \mid \Im(z)=n^k\}$ and $\{z\in \hyp^2 \mid \Im(z)=n^{k+1}\}$ is exactly $\log(n)$. Hence, any $z\in \hyp^2$  has distance at most $ \log(n)$ from  a horocycle $\{z\in \hyp^2 \mid \Im(z)=n^k\}$ for some $k\in \Z$. By the triangle inequality, $z$ has distance at most $ \log(n)+1$ from some point in the orbit of $i$.

This clearly proves that $\bigcup_{g\in BS(1,n)}g B_{\log(n)+1}(i)=\hyp^2$. Therefore by \cite[Lemma~3.11]{ABO}, the action $BS(1,n)\curvearrowright \hyp^2$ is equivalent to the action $BS(1,n)\curvearrowright \Gamma(BS(1,n),S)$ where \[S=\{g\in BS(1,n) \mid d_{\hyp^2}(i,gi)\leq 2\log(n)+3\}.\] We will prove that $[S\cup \{t^{\pm 1}\}]=[Q^- \cup \{t^{\pm 1}\}]$, which will finish the proof.

Let $g\in S$. We may  write  $g=rt^k$ where $r\in H$ and $k\in \Z$. Observe that $d(i,gi)\geq |k|\log(n)$. Since $g\in S$, we have \[|k|\log(n)\leq 2\log(n)+3,\] and hence \[|k|\leq \frac{2\log(n)+3}{\log(n)}\leq \frac{2\log(2)+3}{\log(2)}<7.\] We have $gi=n^ki+r$, and hence \[d(i,gi)=2\operatorname{arcsinh}\left(\frac{1}{2} \sqrt{\frac{r^2+(n^k-1)^2}{n^k}}\right)\geq 2\operatorname{arcsinh}\left(\frac{1}{2} \sqrt{\frac{r^2}{n^7}}\right).\] This clearly defines an upper bound on $|r|$, and hence there exists a uniform $l>0$ (that is, independent of $r$) such that $|\alpha^{-l}(r)|<1$. For such $l$, we clearly have $\alpha^{-l}(r)\in Q^-$.

To summarize, we have $g=rt^k$, where $|k|<7$ and $\alpha^{-l}(r)\in Q^-$. In other words, there exists $s\in Q^-$ such that \[g=\alpha^l(s)t^k=t^lst^{-l}t^k.\]This proves that $\|g\|_{Q^-\cup \{t^{\pm 1}\}}\leq 2l+|k|+1<2l+8$. Thus, \[[Q^-\cup \{t^{\pm 1}\}]\preccurlyeq [S\cup \{t^{\pm 1}\}].\] On the other hand, any element $s\in Q^-$ has $d(i,si)< 1<2\log(n)+3$, so we automatically have $s\in S$. So $Q^-\cup \{t^{\pm 1}\}\subseteq S\cup \{t^{\pm 1}\}$ and this proves \[[S\cup \{t^{\pm 1}\}] \preccurlyeq [Q^- \cup \{t^{\pm 1}\}]. \qedhere\]
\end{proof}

\section{Hyperbolic structures of $BS(1,n)$}
\subsection{Quasi-parabolic structures} \label{section:qpstruct}

\begin{lem} \label{lem:commutator}
The commutator subgroup of $BS(1,n)$ has index $n-1$ in $H$.
\end{lem}

\begin{proof}
The abelianization of $BS(1,n)$ is given by the obvious homomorphism \[BS(1,n)=\langle a, t \mid tat^{-1}=a^n\rangle \to \langle a, t\mid [a,t]=1, a=a^n\rangle = \Z \oplus \Z/((n-1)\Z).\] The kernel of this homomorphism is $[BS(1,n),BS(1,n)]$ whereas $H$ is the kernel of the composition \[BS(1,n)\to \Z\oplus \Z/((n-1)\Z) \to \Z.\] The lemma follows easily from this.
\end{proof}

The proof of Proposition \ref{prop:qpconfining} is a modification to the proofs of \cite[Theorems 4.4 \& 4.5]{Bal} and \cite[Proposition~4.5]{CCMT}. We recall the statement for the reader's convenience:

\qpconfining*

\begin{proof}
Given a strictly confining subset $Q$, a quasi-parabolic structure is constructed in \cite[Proposition~4.6]{CCMT}. 

It remains to prove the forward direction. Let $[T]\in\mathcal H_{qp}(BS(1,n))$. Fix a sequence $\mathbf x=(x_n)$ in $\Gamma(G,T)$, let $q_\mathbf x\colon G\to \R$ be the associated quasi-character, and let $\rho\colon G\to\R$ be the Busemann pseudocharacter (see Section \ref{sec:qpar} for definitions). Since $BS(1,n)$ is amenable, $\rho$ is a homomorphism, and  $\rho(h)=0$ for all $h\in H$ by Lemma \ref{lem:commutator}. Moreover, we must have $\rho(t)\neq 0$, else $\rho(g)=0$ for all $g\in BS(1,n)$.

\begin{claim} There exist constants $r_0,n_0\geq 0$ such that $Q=\bigcup_{i=1}^{n_0-1}\alpha^i(B(1,r_0)\cap H)$ is confining under the action of $\alpha$ or $\alpha^{-1}$.
\end{claim}

\begin{proof}[Proof of Claim]
As $[T]\in\mathcal H_{qp}(BS(1,n))$, the group $BS(1,n)$ fixes a single point $\xi\in\partial\Gamma(BS(1,n),T)$. As $\rho(t)\neq0$, $t$ acts as a hyperbolic isometry of $\Gamma(G,T)$, and thus either $t$ or $t^{-1}$ has $\xi$ as its repelling point. Assume without loss of generality that it is $t$; we will show that $Q$ is strictly confining under the action of $\alpha$. On the other hand, if we assume that $t^{-1}$ has $\xi$ as its repelling point, then an analogous proof will show that $Q$ is strictly confining under the action of $\alpha^{-1}$.

First note that the sequence $(1,t^{-1},t^{-2},\dots)$ defines a $K$-quasi-geodesic ray in $\Gamma(G,T)$ for some $K$, and thus so does the sequence $(g,gt^{-1},gt^{-2},\dots)$ for any $g\in BS(1,n)$. Recall that there is a constant $r_0$, depending only on the hyperbolicity constant of $\Gamma(G,T)$ and $K$ such that any two $K$-quasi-geodesic rays with the same endpoint on $\partial \Gamma(G,T)$ are eventually $r_0$--close to each other. In particular, if $\rho(g)\leq C$, then there is a constant $n_0=n_0(d_T(1,g))$ such that \begin{equation} \label{eqn:distest} d_T(t^{-n},gt^{-n})\leq r_0+C\end{equation} for all $n\geq n_0$.

Fix $n_0=n_0(r_0)$, and define \[Q=\bigcup_{i=0}^{n_0-1}\alpha^i(B(1,r_0)\cap H).\] Choose $r_1$ such that $Q\subseteq B(1,r_1)$. Note that such an $r_1$ exists since for any $i$ and any $h\in B(1,r_0)\cap H$ we have \[\|\alpha^i(h)\|_T=\|t^iht^{-i}\|_T \leq r_0 +2i\|t\|_T.\] Let $n_1=n_0(2r_1)$. 

For any $h\in B(1,r_0)\cap H$, we have $d_T(1,h)\leq r_0$ and $\rho(h)=0$, and so it follows from (\ref{eqn:distest}) that for all $n\geq n_0$, \[d_T(\alpha^n(h),1)=d_T(t^{n}ht^{-n},1)=d_T(ht^{-n},t^{-n})\leq r_0.\]  Therefore, $\alpha^n(h)\in B(1,r_0)\cap H$, and thus for all $n\geq n_0$, \[\alpha^{n}(B(1,r_0)\cap H)\subseteq B(1,r_0)\cap H.\]

We now check the conditions of Definition \ref{def:confining}.
Let $h\in Q$, so that $h\in \alpha^i(B(1,r_0)\cap H)$ for some $0\leq i\leq n_0-1$. If $i<n_0-1$, then $\alpha(h)\in\alpha^{i+1}(B(1,r_0)\cap H)\subseteq Q$. On the other hand, if $i=n_0-1$, then $\alpha(h)\in\alpha^{n_0}(B(1,r_0)\cap H)\subseteq B(1,r_0)\cap H\subseteq Q$. Therefore condition (a) holds.

For any $h\in H$, there is a constant $n_h=n_0(d_T(1,h))$ such that $\alpha^{n_h}(h)\in B(1,r_0)\cap H\subseteq Q$. Therefore $H=\bigcup_{i=0}^\infty \alpha^{-i}(Q)$ and (b) holds.

Finally,  $Q+Q\subseteq B(1,2r_1)$, and so $\alpha^{n_1}(Q+Q)\subseteq B(1,r_0)\cap H\subseteq Q$, and (c) holds with constant $n_1$. 

Therefore, $Q$ is confining under the action of $\alpha$, concluding the proof of the claim.\end{proof}

Let $S=Q\cup\{t^{\pm 1}\}$. We will show that the map $\iota\colon (BS(1,n),d_S)\to (BS(1,n),d_T)$ is a quasi-isometry. Since $\sup_{s\in S}d_T(1,s)<\infty$, the map $\iota$ is Lipschitz. Thus it suffices to show that for any bounded subset $B\subseteq \Gamma(BS(1,n),T)$, we have $\sup_{b\in B}d_S(1,b)<\infty$. Fix any $M>0$ and let $B\subseteq \Gamma(BS(1,n),T)$ be such that $d_T(1,b)\leq M$ for all $b\in B$. For each $b\in B$, we have $b=ht^k$ for some $h\in H$ and some $k$.

% First, notice that $d_X(1,h)\leq M$, and so, as in the proof of the claim, there exists a uniform constant $N=N(M)$ such that $t^Nht^{-N}=\alpha^N(h)\in Q$. Therefore, $d_S(1,h)\leq 2N+1$. 

By the definition of $q_\mathbf x$ (see (\ref{eqn:qchar})), we have $q_\mathbf x(g)\leq d_T(1,g)$, and since there exists a constant $D$ (the defect of $q_\mathbf x$) such that $|\rho(g)-q_\mathbf x(g)|\leq D$, we have \[-d_T(1,g)-D\leq \rho(g)\leq d_T(1,g)+D.\]  Consequently $\rho$ maps bounded subsets of $\Gamma(G,T)$ to bounded subsets of $\R$. Since $\rho(ht^n)=\rho(h)+n\rho(t)=n\rho(t)$, for all $n$, it follows that there is a constant $K$ such that for any $b$ with $d_T(1,b)=d_T(1,ht^k)\leq M$ we have $k\leq K$. This implies that \[d_T(1,h)\leq d_T(1,ht^k)+kd_T(1,t)\leq M+Kd_T(1,t).\] Hence by (\ref{eqn:distest}), there is a uniform $N$ such that $\alpha^N(h)\in Q$ and therefore $d_S(1,h)\leq 2N+1$. Therefore, \[d_S(1,b)=d_S(1,ht^k)\leq d_S(1,h)+k\leq 2N+1+K.\]

Since the map $\iota$ is the identity map on vertices, it is clearly surjective, and therefore it is a quasi-isometry.

Finally, $Q$ is strictly confining under the action of $\alpha$.  Indeed, if $Q$ is confining but not strictly confining under the action of $\alpha$, then $[Q\cup\{t^{\pm1}\}]=[T]\in\mathcal H_\ell(BS(1,n))$, which is a contradiction.
\end{proof}

Recall that given an ideal $\frak a$ of $\Z_n$, there is an associated subset $\mathcal C(\frak a)\subseteq H$ defined in (\ref{eqn:subsetS}) which is confining under the action of $\alpha$. Recall also that $\mathfrak a=\mathfrak a_1\times\cdots \mathfrak a_k$ is \emph{full} if $\mathfrak a_j$ is either $(0)$ or $\Z_{p_j}$ for every $j=1,\dots, k$ (see Definition \ref{def:fullideal}).

\begin{lem} \label{lem:injective}
Let $\frak a,\frak b$ be full ideals of $\Z_n$ such that $\frak a\not\leq\frak b$ and $\frak b\not\leq\frak a$. Then $[\mathcal C(\frak a)\cup\{t^{\pm 1}\}]$ and  $[\mathcal C(\frak b)\cup\{t^{\pm1}\}]$ are incomparable.
\end{lem}

\begin{proof}
Since $\frak a=\frak a_1\times\cdots\times\frak a_k$ and $\frak b=\frak b_1\times\cdots\times\frak b_k$ are full ideals of $\Z_n$,  there exist $1\leq i,j\leq k$ such that $\frak a_i=(0)$, $\frak a_j=\Z_{p_j^{n_j}}$, $\frak b_i=\Z_{p_i^{n_i}}$, and $\frak b_j=(0)$. 

Consider $\mathcal C(\frak a)$ and $\mathcal C(\frak b)$.  For any $x=\pm x_r\cdots x_0.x_{-1}\cdots x_{p(x)}\in \mathcal C(\frak a)$, there is an element $a=\dots x_{-1}\dots x_{p(x)}\in \frak a$. Similarly, for any $y=\pm y_m\cdots y_0.y_{-1}\cdots y_{p(y)}\in\mathcal C(\frak b)$, there is an element $b=\dots y_{-1}\dots y_{p(y)} \in\frak b$. Since $\frak a_i=(0)$, \[x_{p(x)}\equiv 0\mod p_i^{n_i},\] and since $\frak b_j=(0)$, \[y_{p(y)}\equiv 0\mod p_j^{n_j},\] by Lemma \ref{lem:partialsums}.

For any $K\geq 0$, choose $x\in\mathcal C(\frak a)$ such that $p(x)=-K$ and $c(x)\not\equiv 0\mod p_j^{n_j}$, which is possible since $\frak a_j=\Z_{p_j^{n_j}}$. By Lemma \ref{lem:movegenerators} we can find an expression of $x$ as a minimal length word in the generating set $\mc C(\frak b)\cup\{t^{\pm1}\}$ of the form  \[x=t^{-u}(g_1+\cdots+g_w)t^u.\]  By Lemma \ref{lem:confiningsubgroup} we may write $g_1+\dots+g_w=g\in \mathcal C(\frak b)$, and the result is that $x=\alpha^{-u}(g)$. Now, since $g\in\mathcal C(\frak b)$ and $\frak b_j=(0)$, it follows that if $p(g)<0$, then $c(g)\equiv 0\mod p_j^{n_j}$. But this implies that $c(x)=c(\alpha^{-u}(g))\equiv 0\mod p_j^{n_j}$, which is a contradiction. Consequently, we must have $p(g)\geq 0$, which implies that $u\geq K$. Therefore \[\|x\|_{\mathcal C(\frak b)\cup\{t^{\pm 1}\}}\geq 2K+1.\]  Since $K$ was arbitrary, it follows that \[\sup_{x\in\mathcal C(\frak a)\cup\{t^{\pm1}\}}\|x\|_{\mathcal C(\frak b)\cup\{t^{\pm 1}\}}=\infty,\] and thus \[[\mathcal C(\frak a)\cup\{t^{\pm 1}\}]\not \preccurlyeq [\mathcal C(\frak b)\cup\{t^{\pm1}\}].\]

A similar argument shows that \[[\mathcal C(\frak a)\cup\{t^{\pm 1}\}]\not \succcurlyeq [\mathcal C(\frak b)\cup\{t^{\pm1}\}],\] and the result follows.
\end{proof}

\begin{lem} \label{lem:orderreversing}
Let $\frak a,\frak b$ be full ideals of $\Z_n$, and suppose $\frak a\lneq\frak b$. Then $[\mathcal C(\frak a)\cup\{t^{\pm 1}\}]\succnsim [\mathcal C(\frak b)\cup\{t^{\pm1}\}]$.
\end{lem}

\begin{proof}
Since $\frak a\lneq \frak b$ and $\frak a=\frak a_1\times \cdots\times\frak a_k,\frak b=\frak b_1\times\cdots\times \frak b_k$ are both full ideals of $\Z_n$, we have $\frak a\subsetneq \frak b$. Therefore, there exists some $1\leq i\leq k$ such that $\frak a_i=(0)$ while $\frak b_i=\Z_{p_i^{n_i}}$. Consequently, for all $a=\dots a_2a_1a_0\in\frak a$, we have $a_0\equiv 0\mod p_i^{n_i}$ by Lemma \ref{lem:partialsums}.  

We first show that $\mathcal C(\frak a)\subseteq \mathcal C(\frak b)$. Let $x=\pm x_r\cdots x_1x_0.x_{-1}\cdots x_{p(x)}\in\mathcal C(\frak a)$. Then  there exists some $a\in \frak a$ such that $a=\dots x_{-1}\dots x_{p(x)}$. Since $\frak a\subseteq \frak b$, we have $a\in\frak b$, and so by the definition of $\mathcal C(\frak b)$, it follows that $x\in\mathcal C(\frak b)$. Therefore $C(\frak a)\cup\{t^{\pm 1}\}\subseteq \mathcal C(\frak b)\cup\{t^{\pm1}\}$, and  \[[\mc C(\frak a)\cup\{t^{\pm 1}\}]\succcurlyeq [\mathcal C(\frak b)\cup\{t^{\pm1}\}].\]

Finally, an argument analogous to the proof of Lemma \ref{lem:injective} shows that $[\mathcal C(\frak a)\cup\{t^{\pm 1}\}]\neq[\mathcal C(\frak b)\cup\{t^{\pm1}\}]$, completing the proof.
\end{proof}

Recall that given a subset $Q\subseteq H$ which is confining under the action of $\alpha$, there is an associated ideal $\mathcal I(Q)$ of $\Z_n$ defined by (\ref{eqn:idealL}).

\begin{lem} \label{lem:TequivSLT}
For any confining subset $Q$ under the action of $\alpha$, we have $[Q\cup\{t^{\pm1}\}]= [\mathcal C(\mathcal I(Q))\cup\{t^{\pm1}\}]$. 
\end{lem}

\begin{proof}
By Lemma \ref{lem:SLTsubsetT}, we have that $\mathcal C(\mathcal I(Q)) \subseteq \alpha^{-M}(Q)$ for some $M$. Note that this implies that \[[Q\cup\{t^{\pm1}\}]\preccurlyeq [\mathcal C(\mathcal I(Q))\cup\{t^{\pm1}\}].\]   We will show that $[\mathcal C(\mathcal I(Q))\cup\{t^{\pm1}\}]\preccurlyeq [Q\cup\{t^{\pm1}\}]$. Suppose this is not the case. Then $Q\not\subseteq \alpha^{-k}(\mathcal{C}(\mathcal{I}(Q)))$ for any $k$ and hence there exist elements \[a=a_r \cdots a_0. a_{-1} \cdots a_{-s}\in Q\] with \[\inf \{k \mid \alpha^k(a)\in \mathcal{C}(\mathcal{I}(Q))\}\] arbitrarily large. Let $a$ be an element as above, let \[\ell=\inf \{k \mid \alpha^k(a)\in \mathcal{C}(\mathcal{I}(Q))\},\] and assume $\ell>k_0$, where $k_0$ is large enough that $\alpha^{k_0}(Q+Q)\subseteq Q$. Let $t\leq s$ be largest with the property that there does not exist an element of the form \[\dots a_{-t} \dots a_{-s} \in \mathcal{I}(Q).\] Note that $t> \ell$, for otherwise we would have \[\alpha^{\ell-1}(a)= a_r \cdots a_0 a_{-1} \cdots a_{-\ell+1}.a_{-\ell} \cdots a_{-s}\in \mathcal{C}(\mathcal{I}(Q)),\] contradicting the definition of $\ell$ as an infimum. We consider two cases:

\begin{enumerate}[(1)]
\item If $t=s$, then by definition there does not exist an element of $\mathcal{I}(Q)$ with one's digit $a_{-s}$. In other words, there does not exist an element of the form $\ldots a_{-s}$ in $\mathcal{I}(Q)$.

\item On the other hand suppose that $t<s$. Then by definition of $t$, there exists an element \[x= \dots x_2 x_1 x_0 a_{-t-1} \dots a_{-s}\in \mathcal{I}(Q).\] Let $y=\dots y_2 y_1 y_0\in \mathcal{I}(Q)$ be the additive inverse of $x$. That is, \[\begin{array}{l l l l l l}
& \cdots & x_0 & a_{-t-1} & \cdots & a_{-s} \\
+ & \cdots & y_{s-t} & y_{s-t-1} & \cdots & y_0 \\
\hline
& \cdots & 0 & 0 & \cdots & 0 \\\end{array}\] By definition of $\mathcal{I}(Q)$,  there exists an element \[b=b_u \cdots b_0.y_{s-1} y_{s-2} \cdots y_0 \in Q.\] We have $c=a+b\in Q+Q$, where $c=c_v\cdots c_0.c_{-1} \cdots c_{-t}$ is given by \[\begin{array}{l l l l l l l l l l l l l}
 & & & & a_r & \cdots & a_0 . & a_{-1} & \cdots & a_{-t} & a_{-t-1} & \cdots & a_{-s} \\
+  & & b_u & \cdots & b_r & \cdots & b_0 . & y_{s-1} & \cdots & y_{s-t} & y_{s-t-1} & \cdots & y_0 \\
\hline
c_v & \cdots & c_u & \cdots & c_r & \cdots & c_0. & c_{-1} & \cdots & c_{-t} & 0 & \cdots & 0 \\
 \end{array}\] (note, we are assuming in the above expression that $u\geq r$, but the case $u< r$ is identical). Therefore \[\alpha^{k_0}(c)=c_v \cdots c_0 c_{-1} \cdots c_{-k_0}.c_{-k_0-1} \cdots c_{-t} \in Q.\] Note that there does not exist an element of $\mathcal{I}(Q)$ whose one's digit is $c_{-t}$. For suppose that $z=\dots z_2 z_1 z_0 c_{-t}$ is such an element. Then we also have \[n^{s-t}z= \dots z_2 z_1 z_0 c_{-t} \underbrace{0 \dots \dots \dots 0}_{s-t\text{ times}}\in \mathcal{I}(Q)\] and hence $n^{s-t}z-y\in\mathcal{I}(Q)$ is given by \[\begin{array}{l l l l l l}
 & \cdots & c_{-t} & 0 & \cdots & 0 \\
- & \cdots & y_{s-t} & y_{s-t-1}& \cdots & y_0 \\
\hline 
& \cdots & a_{-t} & a_{-t-1} & \cdots & a_{-s} \\
  \end{array}\] contradicting the definition of $t$.
  
 \end{enumerate}

Taking $d=a$ in Case (1) above or $d=\alpha^{k_0}(c)=c_v \cdots c_{-k_0}. c_{-k_0-1} \cdots c_{-t}$ in Case (2) above, we have shown so far that there are elements $d_w \cdots d_0.d_{-1} \cdots d_{-u}\in Q$ with $u$ arbitrarily large (at least $ \ell-k_0$) and with the property that there does not exist any element of the form $\ldots d_{-u}$ in $\mathcal{I}(Q)$. That is, there exists a sequence $u_i \to \infty$ and a sequence $\{d^i\}_{i=1}^\infty \subset Q$ with $p(d^i)=-u_i$, and $c(d^i)=d^i_{-u_i}$ with the property that there does not exist an element of the form $\ldots d^i_{-u_i}$ in $\mathcal{I}(Q)$. Writing $d^i=d_{w_i}^i \cdots d_0^i . d_{-1}^i \cdots d_{-u_i}^i$ we may pass to a subsequence to assume that the sequence of integers $d_{-1}^i \cdots d_{-u_i}^i \in \Z\subseteq \Z_n$ converges to a number $\ldots e_2 e_1 e_0\in \Z_n$.

We claim that $\ldots e_2 e_1 e_0\in \mathcal{I}(Q)$. To prove the claim, note that given any $t\geq 0$ and all large enough $i$, the number $d_{-1}^i \cdots d_{-u_i}^i$ has \[d^i_{-u_i}=e_0,\,\,\, d^i_{-u_i+1}=e_1,\,\,\, \ldots\,\,\, d^i_{-u_i+t}=e_t.\] We have \[\begin{array}{l l l} \alpha^{u_i-t-1}(d^i) & =& d_{w_i}^i \cdots d_0^i d_{-1}^i \cdots d_{-u_i+t+1}^i . d_{-u_i+t}^i d_{-u_i+t-1}^i \cdots d_{-u_i}^i \\ &= & d_{w_i}^i \cdots d_{-u_i+t+1}^i. e_t e_{t-1} \cdots e_0\end{array}\]  for all such $i$. Since $\alpha^{u_i-t-1}(d^i)\in Q$ and $t$ is arbitrary, this proves the claim. However, this is a contradiction, as we have $d_{-u_i}^i=e_0$ for all large enough $i$, while there does not exist an element of the form $\ldots d_{-u_i}^i$ in $\mathcal{I}(Q)$ for any $i$. This completes the proof.
\end{proof}

Define $\mathcal J_n$ to be the poset $2^{\{1,\dots, k\}}\setminus\{\{1,\dots, k\}\}$ with the partial order given by inclusion. Recall that $\mathcal K_n$ is the poset $2^{\{1,\ldots k\}} \setminus \{\emptyset\}$. We have that $\mathcal K_n$ is isomorphic to the \textit{opposite poset} of $\mathcal J_n$.

We now define a map \[\Phi\colon\mathcal J_n\to \mathcal H_{qp}(BS(1,n))\] as follows. Given $A\in\mathcal J_n$, let $\frak a=\frak a_1\times \ldots \times \frak a_k$ be the full ideal of $\Z_n$ defined by $\frak a_i=(0)$ if and only if $i\in A$, and let \begin{equation}\label{eqn:Phi} \Phi(A)=[\mathcal C(\frak a)\cup\{t^{\pm 1}\}].\end{equation}

\begin{prop}\label{prop:qpstructure}
The map $\Phi\colon \mathcal J_n\to \mathcal H_{qp}(BS(1,n))$ defined in (\ref{eqn:Phi}) is an order-reversing injective map. Hence $\Phi$ induces an injective homomorphism of posets \[\mathcal K_n \to \mathcal H_{qp}(BS(1,n)).\] Moreover, $H_{qp}(BS(1,n))$ contains exactly one additional structure which is incomparable to every $[Y]\in \Phi(\mathcal J_n)$.
\end{prop}

\begin{proof}
Lemmas \ref{lem:injective} and  \ref{lem:orderreversing} show that the map $\Phi$ is an injective order-reversing map of posets.

By Proposition \ref{prop:qpconfining}, if $[S]\in \mathcal{H}_{qp}(BS(1,n))$, then there exists a $Q\subseteq H$ which is strictly confining under the action of $\alpha$ or $\alpha^{-1}$ and such that $[S]=[Q\cup \{t^{\pm 1}\}]$. 

Fix $[S]\in H_{qp}(BS(1,n))$ such that the corresponding subset $Q$ is strictly confining under the action of $\alpha$, and consider the ideal $\mathcal I(Q)$. By Lemma \ref{lem:TequivSLT}, $[S]= [\mathcal C(\mathcal I(Q))\cup \{t^{\pm1}\}]$. Moreover, by Lemma \ref{lem:equivalentfull}, there is a proper full ideal $\frak a\sim\mathcal I(Q)$, and $\mathcal C(\mathcal I(Q))=\mathcal C(\frak a)$ by Lemma \ref{lem:equividealsequalsubsets}. Thus $[S]= [\mathcal C(\frak a)\cup \{t^{\pm1}\}]$. Let $A=\{1\leq i\leq k\mid \frak a_i=(0)\}\subseteq \mathcal J_n$. Then $[S]=\Phi(A)$, and so every quasi-parabolic structure whose associated subset is strictly confining under the action of $\alpha$ is in the image of $\Phi$.

By Proposition \ref{prop:singlealpha^-1}, $\mathcal H_{qp}(BS(1,n))$ has a single additional element, $[Q^-\cup\{t^{\pm1}\}]$, where $Q^-$, defined in (\ref{eqn:Q-}), is strictly confining under the action of $\alpha^{-1}$. It remains to show that $[Q^-\cup\{t^{\pm1}\}]$ is incomparable to all $[S]\in \mathcal H_{qp}(BS(1,n))\setminus\{[Q^-\cup\{t^{\pm1}\}]\}$.

To see this last fact, note that the action $BS(1,n)\curvearrowright \Gamma(BS(1,n),Q^-\cup \{t^{\pm 1}\})$ is equivalent to the action $BS(1,n)\curvearrowright \hyp^2$ by Proposition \ref{prop:planeaction}. Hence in this action, the common fixed point of all elements of $BS(1,n)$ is the \textit{attracting} fixed point of $t$. On the other hand, for $[S]\in \mathcal H_{qp}(BS(1,n))\setminus\{[Q^-\cup\{t^{\pm1}\}]\}$, $BS(1,n)\curvearrowright \Gamma(BS(1,n),S)$ is equivalent to the action of $BS(1,n)$ on one of the Bass-Serre trees described in Section \ref{section:bass-serre}. Hence in the action $BS(1,n)\curvearrowright \Gamma(BS(1,n),S)$, the common fixed point of all elements of $BS(1,n)$ is the \textit{repelling} fixed point of $t$. If we had, for example, $[Q^-\cup \{t^{\pm 1}\}]\preccurlyeq [S]$ then this would imply that every element of $BS(1,n)$ would fix the repelling fixed point of $t$ in $\partial \Gamma(BS(1,n), Q^- \cup \{t^{\pm 1}\})$ as well as the attracting fixed point of $t$. This would imply that the action $BS(1,n)\curvearrowright \Gamma(BS(1,n), Q^- \cup \{t^{\pm 1}\})$ is lineal, which is a contradiction.
\end{proof}

\subsection{Proof of Theorem \ref{thm:BS1nstructure}}  
Proposition \ref{prop:qpstructure} gives a complete description of $\mathcal H_{qp}(BS(1,n))$. We now turn our attention to other hyperbolic structures.

We first show that for any $n\geq 2$, $|\mathcal{H}_l(BS(1,n))|=1$. Consider an action $BS(1,n)\curvearrowright X$ with $X$ hyperbolic. Then every element of $H$ must act elliptically or parabolically. For if $g\in H$ then $tgt^{-1}$ and $g$ have the same stable translation length. However, $tgt^{-1}=g^n$ has stable translation length equal to $n$ times the stable translation length of $g$. This is only possible if the stable translation length of $g$ is 0. Hence the induced action of $H$ on $X$ is either elliptic or parabolic. 
Since $H\trianglelefteq BS(1,n)$, if $H\curvearrowright X$ is parabolic, then every element of $BS(1,n)$ must fix the unique point in $\partial X$ which is fixed by $H$,  hence $BS(1,n)\curvearrowright X$ is parabolic or quasi-parabolic.  In particular, if $BS(1,n)\curvearrowright X$ is a lineal action then $H\curvearrowright X$ is elliptic. This shows that if $[S] \in \mathcal H_l(BS(1,n))$ then $[S]\preccurlyeq [H\cup \{t^{\pm 1}\}]$. But by \cite[Theorem~4.22]{ABO}, if $[A]\preccurlyeq [B]$ and both structures are lineal, then $[A]=[B]$. Therefore $[S]=[H\cup \{t^{\pm 1}\}]$.

Every quasi-parabolic structure dominates the lineal structure defined by its Busemann quasi-morphism. Since $|\mathcal H_l(BS(1,n))|=1$, it follows that every element of $\mathcal H_{qp}(BS(1,n))$ dominates this single lineal structure.

For any $n\geq 2$,  $BS(1,n)$ is solvable, and so contains no free subgroups. Thus by the Ping-Pong lemma, $H_{gt}(BS(1,n))=\emptyset$.

Finally, for any group $G$, $\mathcal H_e(G)$ has a single element which is the smallest element in $\mathcal H(G)$, completing the proof of Theorem \ref{thm:BS1nstructure}.

%%%%%%%%%%%%%%%%%%%%%%%%%%%%%%%
\section{Generating confining subsets} \label{section:generating}

Consider a group $G=H\rtimes \Z$ where $\Z=\langle t\rangle$ acts by $tht^{-1}=\alpha(h)$ for any $h\in H$, where $\alpha\in \Aut(H)$.  In this section, we give a general method for constructing confining subsets of such groups.

Let $S$ be a symmetric subset of $H$ with the following properties:
\begin{enumerate}[(i)]
\item $\alpha(S)\subseteq S$
\item $\bigcup_{n\geq 0} \alpha^{-n}(S)$ generates $H$.
\end{enumerate}

Define a subset $Q^i\subseteq H$ by \[Q^i=Q^i_0\cup Q^i_1\cup Q^i_2\cup \ldots\] where $Q^i_0=S$ and \[Q^i_{n+1}=Q^i_n\cup \alpha^i(Q^i_n\cdot Q^i_n)\] for all $n\geq 0$. In other words, $Q^i$ is the smallest subset of $H$ containing $S$ and with the property that $\alpha^i(Q^i \cdot Q^i)\subset Q^i$.

\begin{prop}\label{lem:Q^i}
$Q^i$ is a confining subset of $H$ under the action of $\alpha$.
\end{prop}

\begin{proof}
First, we prove by induction that $\alpha(Q^i_n)\subseteq Q^i_n$ for all $n$. The base case $n=0$ holds by point (i). Suppose for induction that $\alpha(Q^i_n)\subseteq Q^i_n$. Now if $x\in Q^i_{n+1}$ then $x\in Q^i_n$ or $x\in \alpha^i(Q^i_n\cdot Q^i_n)$. In the first case we have $\alpha(x)\in Q^i_n\subseteq Q^i_{n+1}$. Otherwise, we may write $x=\alpha^i(yz)$ where $y,z\in Q^i_n$. Then we have \[\alpha(x)=\alpha^i(\alpha(y)\alpha(z))\] and since $\alpha(y),\alpha(z)\in Q^i_n$ we have $\alpha(x)\in \alpha^i(Q^i_n\cdot Q^i_n)\subseteq Q^i_{n+1}$.

Since $\alpha(Q_n^i)\subset Q_n^i$ for all $n$, we have $\alpha(Q^i)\subset Q^i$.

Now to see that $\alpha^i(Q^i\cdot Q^i)\subseteq Q^i$ simply use the fact that $\alpha^i(Q^i_n\cdot Q^i_n)\subseteq Q^i_{n+1}$ and $Q^i$ is the union of the sets $Q^i_n$ for $n\geq 0$.

Finally, we prove that $H=\bigcup_{n\geq 0} \alpha^{-n}(Q^i)$. We use fact (ii) above and induction on the word length of an element $h\in H$ in the semigroup generating set $\bigcup_{n\geq 0}\alpha^{-n}(S)$. If $h$ has word length one with respect to this generating set then $h=\alpha^{-n}(s)$ for some $s\in S$. Hence, $\alpha^n(h)=s\in Q^i_0$. Suppose for induction that if $h$ has word length at most $k$, then $\alpha^n(h)\in Q^i$ for some $n\geq 0$. Let $h$ have word length $k+1$. Then we may write $h=\alpha^{-n}(s)h'$ where $s\in Q^i$ and $h'$ has word length $k$. By induction there exists $m$ such that $\alpha^m(h')\in Q^i$. We have \[\alpha^{m+n}(h)=\alpha^m(s)\alpha^{m+n}(h').\] Since $\alpha(Q^i)\subseteq Q^i$ and $\alpha^m(h')\in Q^i$, we have $\alpha^{m+n}(h')\in Q^i$. Similarly $\alpha^m(s)\in Q^i$. Choose $r$ large enough such that $\alpha^m(s),\alpha^{m+n}(h')\in Q^i_r$. Then we have \[\alpha^{m+n+i}(h)=\alpha^i(\alpha^m(s)\alpha^{m+n}(h'))\in \alpha^i(Q^i_r \cdot Q^i_r)\subseteq Q^i_{r+1}.\] This completes the induction.
\end{proof}

%%%%%%%%%%%%%%%%%%%%%%%%%%%%%%%%%%%%%%%%%%%%%%%%%%%%%%%%%%%%%%%%%%%

By \cite[Proposition~4.6]{CCMT}, $[Q^i\cup \{t^{\pm 1}\}]\in \H_{qp}(G)\cup\H_\ell(G)$ for all $i$. Clearly we have 
\begin{equation}\label{eqn:chain}
[Q^1\cup \{t^{\pm 1}\}] \preccurlyeq [Q^2\cup \{t^{\pm 1}\}] \preccurlyeq [Q^3\cup \{t^{\pm 1}\}] \preccurlyeq \ldots.
\end{equation} 
In general, there is no reason that any of these inequalities need be strict.   However, for certain groups they are \emph{all} strict.  In the following proposition, we apply Proposition~\ref{lem:Q^i} to the lamplighter group $\Z\wr\Z\simeq \Z\left[x,\frac{1}{x}\right]\rtimes \Z$ to demonstrate this phenomenon.

\begin{prop}\label{prop:lamplighter}
Let $H=\Z\left[x,\frac{1}{x}\right]$ (the additive group of the Laurent polynomial ring over $\Z$) where $t$ acts on $H$ by $t p(x) t^{-1}=xp(x)$. Set $S=\{\pm1 ,\pm x,\pm x^2,\ldots\}$. Define $Q^i$ using $S$ as above. Then we have $[Q^i\cup \{t^{\pm 1}\}] \neq [Q^j\cup \{t^{\pm 1}\}]$ for $i<j$. Moreover, $[Q^i \cup \{t^{\pm 1}\}]\in \mathcal{H}_{qp}(H\rtimes \Z)$ for each $i$.
\end{prop}

Before turning to the proof, we note that the existence of a countable chain of quasi-parabolic structures for this group follows from \cite[Theorem~1.4]{Bal}, which also gives additional information about the structure of $\mathcal H_{qp}(\Z\wr\Z)$, though it does not give a complete description.   In fact, one can see that such a chain exists by fixing any $m\in \Z_{>0}$ and  considering the sequence of quotients 
\[
\Z\wr \Z\onto \cdots \onto (\Z/m^3\Z) \wr \Z \onto (\Z/m^2\Z) \wr\Z \onto (\Z/m\Z)\wr \Z.
\]
By expressing each $(\Z/m^n\Z)\wr \Z$ as an HNN extension, we obtain a quasi-parabolic action of $(\Z/m^n\Z)\wr \Z$ on the associated Bass--Serre tree.  Since $\Z\wr \Z$ surjects onto $(\Z/m^n\Z)\wr \Z$, we see that $\Z\wr \Z$ acts on this Bass--Serre tree, as well.  From the sequence of quotients, it is clear that the collection of such actions will form a countable chain in $\mc H_{qp}(\Z\wr \Z)$.   This countable chain is equivalent to one constructed by Balasubramanya in \cite[Theorem~1.4]{Bal} associated to the nested subgroups $\cdots \leq m^n\Z \leq m^{n-1}\Z \leq \cdots \leq m\Z \leq \Z$.  In contrast, we do not expect that the countable chain described in Proposition~\ref{prop:lamplighter} corresponds to any chain constructed in~\cite{Bal}.

\begin{proof}[Proof of Proposition~\ref{prop:lamplighter}]
Note the following, which can be proven inductively:

\begin{enumerate}[(i)]
\item $Q^i_r$ consists of polynomials in $x$ (with no terms of negative degree),
\item if $p(x)\in Q^i_r \setminus Q^i_{r-1}$ then every term of $p(x)$ has degree at least $ri$,
\item the largest coefficient of a term of $p(x)\in Q^i_r$ is $2^r$.
\end{enumerate}

Hence we have the following table:

\noindent \begin{tabular}{l | l l l l l l l l l l l l l l l l}
Monomial & $1$ & $x$ & $x^2$ & $\ldots$ & $x^{i-1}$ & $x^i$ & $x^{i+1}$ & $\ldots$ & $x^{2i-1}$ & $x^{2i}$ & $x^{2i+1}$ & $\ldots$ & $x^{ri-1}$ & $x^{ri}$ & $x^{ri+1}$ & $\ldots$ \\
\hline
Largest coeff. & 1 & 1 & 1 & $\ldots$ & 1 & 2 & 2 & $\ldots$ & 2 & 4 & 4 & $\ldots$ & $2^{r-1}$ & $2^r$ & $2^r$ & $\ldots$
\end{tabular}

Here the entry under $x^k$ denotes the \textit{largest absolute value} of the coefficient of the degree $k$ term of any polynomial $p(x)\in Q^i$. A similar table holds for $Q^j$. In particular, we see that if $p(x)\in Q^j$ and $p(x)$ contains a term of degree $k$, then the absolute value of the coefficient of $x^k$ is at most $2^{\lfloor k/j\rfloor}$.

Note, in particular,a that the sequence \[1,\, 2x^i=\alpha^i(1+1),\, 4x^{2i}=\alpha^i(2x^i+2x^i),\, 8x^{3i}=\alpha^i(4x^{2i}+4x^{2i}),\ldots\] is contained in $Q^i$. Hence all have word length $1$ in the generating set $Q^i\cup \{t\}$. We claim that the word length of $2^r x^{ri}$ in the generating set $Q^j\cup \{t\}$ goes to infinity as $r\to \infty$. Let $\|\cdot\|_j$ denote word length with respect to the generating set $Q^j\cup \{t^{\pm 1}\}$.

To prove the claim, write \[2^rx^{ri}=g_1\ldots g_n\] where $g_1\ldots g_n$ is a word in $Q^j\cup \{t^{\pm 1}\}$ of minimal length $n=\|2^r x^{ri}\|_j$ representing $2^rx^{ri}$. Each $g_i$ is either $t$, $t^{-1}$, or a polynomial $p(x)\in Q^j$. By Lemma \ref{lem:movegenerators}, we may rewrite $2^r x^{ri}$ as a word of length $n$ of the form 
 \[2^rx^{ri}=t^{-k} (p_1(x)+\ldots+ p_m(x))t^l.\] Since the word on the right represents a polynomial, we must in fact have $k=l$ and we have \[2^rx^{ri}=\alpha^{-k}(p_1(x)+\ldots+p_m(x)) ,\] and therefore \[p_1(x)+\ldots +p_m(x) = 2^rx^{ri+k}.\] It follows that $n=2k+m$. Since each $p_*(x)$ contains $x^{ri+k}$ as a term with coefficient at most $2^{\lfloor (ri+k)/j\rfloor}$, we must have \[m \geq \frac{2^r}{2^{\lfloor (ri+k)/j\rfloor}}=2^{r-\lfloor (ri+k)/j\rfloor}\geq 2^{r-(ri+k)/j}=2^{(1-i/j)r-k/j}.\]

So to bound $n$ from below, it suffices to minimize $2k+m$ subject to the condition \[m=2^{(1-i/j)r-k/j}.\] Rewriting $2k+m$ in terms of $k$ yields \[2k+2^{(1-i/j)r-k/j}.\] Defining a function $f(k)=2k+2^{(1-i/j)r-k/j}$, we see that $f$ has a unique minimum at the unique zero of its derivative. The derivative with respect to $k$ is \[f'(k)=2-\frac{1}{j}\ln(2)2^{(1-i/j)r-k/j}.\] Solving the equation $f'(k)=0$ for $k$ yields \[k=(j-i)r-j\log_2\left(\frac{j}{\ln(2)}\right) -j.\] Since $j-i>0$ we must have $k\to \infty$ as $r\to \infty$ and, in particular, $n\to \infty$ as $r\to \infty$. In other words, $\|2^rx^{ri}\|_j\to \infty$ as $r\to \infty$.

For the final sentence, simply note that each $Q^i$ is strictly confining since $1\notin \alpha(Q^i)$ for any $i$.
\end{proof}

%%begin{rem}
The argument in the proof of Proposition~\ref{prop:lamplighter} does not work for the wreath product $\Z/n\Z\wr \Z\cong\Z/n\Z\left[x,\frac1x\right]\rtimes \Z$ since the generator of $\Z/nZ$ doesn't have infinite order.   Balasubramanya shows that $\Z/n\Z\wr \Z$ has only finitely many quasi-parabolic structures \cite[Theorem~1.4]{Bal}, hence only finitely many of the inequalities in \eqref{eqn:chain}  can be strict. %non-equivalent structures $[Q^i\cup\{t^{\pm 1}\}]$ for $\Z/n\Z\wr \Z$.
%\end{rem}

%%%%%%%%%%%%%%%%%%%%%%%%%%%%%%

\vspace{1cm}

\noindent \textbf{Carolyn R. Abbott } Department of Mathematics, Brandeis University, Waltham, MA 02453. \\
E-mail: \emph{cra2112@columbia.edu}

\bigskip

\noindent \textbf{Alexander J. Rasmussen } Department of Mathematics, University of Utah, Salt Lake City, UT 84112. \\
E-mail: \emph{rasmussen@math.utah.edu}

\end{document}